\documentclass[amssymb,amsfonts,reqno]{amsart}
\usepackage{color}
\usepackage{latexsym}
\usepackage{amsmath}
\usepackage{amssymb}
\usepackage{amsfonts}
\usepackage{esint}
\usepackage{graphicx}
\usepackage{accents}

\newcommand{\der}{\nabla}

\newcommand{\les}{\lesssim}
\newcommand{\bea}{\begin{eqnarray}}
\newcommand{\eea}{\end{eqnarray}}

\def\pa{\partial}

\def\beaa{\begin{eqnarray*}}
\def\eeaa{\end{eqnarray*}}
\def\pa{\partial}
\def\a{{\alpha}}
\def\b{{\beta}}

\def\si{\sigma}

\def\om{\omega}
\def\Om{\Omega}

\def\th{\theta}

\def\S{{\bf S}}

\def\g{{\bf g}}

\def\nn{{\mathbb N}}

\def\R{{\mathbb R}}

\def\cL{{\mathcal L}}

\def\cO{{\mathcal O}}

\def\qed{$\Box$\medskip}

\def\12{\frac{1}{2}}

\def\one{\bbbone}
\def\cH{{\mathcal H}}

\def\cB{{\mathcal B}}

\def\N{{\mathcal N}}
\def\S{{\mathcal S}}

\def\bep{\begin{proposition}}
\def\eep{\end{proposition}}

\def\4{\frac{1}{4}}

\def\cE{\mathcal{E}}

\def\qed{$\Box$\medskip}

\def\12{\frac{1}{2}}

\def\one{\bbbone}
\def\cH{{\mathcal H}}

\def\N{\nn}
\def\S{{\mathcal S}}

\def\bep{\begin{proposition}}
\def\eep{\end{proposition}}

\def\qed{\hfill \raisebox{0.5ex}{\framebox[1.6ex]{
                                       \rule[0ex]{0ex}{0.3ex} }}}

\def\build#1_#2^#3{\mathrel{\mathop{\kern 0pt#1}\limits_{#2}^{#3}}}
\def\4{\frac{1}{4}}

\def\cE{\mathcal{E}}

\def\<{\langle}
\def\>{\rangle}
\def\one{{\mathchoice {\rm 1\mskip-4mu l} {\rm 1\mskip-4mu l} {\rm 1\mskip-4.5mu l} {\rm
1\mskip-5mu l}}}

\theoremstyle{plain}
\newtheorem{theorem}{Theorem}
\newtheorem{proposition}{Proposition}
\newtheorem{lemma}{Lemma}
\newtheorem{corollary}{Corollary}

\theoremstyle{remark}
\newtheorem{remark}{Remark}

\theoremstyle{definition}
\newtheorem{definition}{Definition}
\numberwithin{equation}{section}
\numberwithin{proposition}{section}
\numberwithin{definition}{section}
\numberwithin{lemma}{section}
\numberwithin{corollary}{section}
\numberwithin{remark}{section}
\begin{document}
\title[Yang-Mills fields on the Schwarzschild black hole]{The decay of the $SU(2)$ Yang-Mills fields on the Schwarzschild black hole for
  spherically symmetric small energy initial data}
\author{Sari Ghanem, Dietrich H\"afner}
\address{Universit\'e Grenoble-Alpes, Institut Fourier, 100 rue des
  maths, 38610 Gi\`eres, France}
\email{Sari.Ghanem@univ-grenoble-alpes.fr}
\email{Dietrich.Hafner@univ-grenoble-alpes.fr}
\maketitle

\begin{abstract} 
We prove uniform decay estimates in the entire exterior of the
Schwarzschild black hole for gauge invariant norms on the
Yang-Mills fields valued in the Lie algebra associated to the Lie
group $SU(2)$. We assume that the initial data are spherically
symmetric satisfying a certain Ansatz, and have small energy, which
eliminates the stationary solutions which do not decay. In
  particular, there don't exist any Coulomb type solutions satisfying
  this Ansatz. We first prove
a Morawetz type estimate for the Yang-Mills fields within this setting, using the Yang-Mills equations directly. We then adapt the proof
constructed in previous work by the first author to show local energy decay and uniform decay of the $L^{\infty}$ norm of the middle components in the entire exterior of the Schwarzschild black hole, including the event horizon.
\end{abstract}

\setcounter{page}{1}
\pagenumbering{arabic}

\section{Introduction}

\subsection{General introduction}

We study the $SU(2)$ Yang-Mills equations on the Schwarzschild
metric, with spherically symmetric initial data. The Yang-Mills fields then take values in $su(2)$, the Lie algebra
associated to the Lie group. We consider initial data which are
spherically symmetric satisfying a certain Ansatz (see
e.g. \cite{BRZ}, \cite{GuHu}, \cite{W}). This Ansatz, which is very frequently used in the literature, is usually called
  ``purely magnetic", and it excludes in particular Coulomb type
  solutions. Nevertheless, other stationary solutions exist within this
  Ansatz (see
e.g. \cite{BRZ} for a description of these stationary
solutions). These stationary solutions are excluded by our small energy
  assumption. We prove for solutions of the
Yang-Mills equations, generated from such an initial data, a Morawetz
type estimate. We can then use the method of \cite{G2} to
prove decay of the fields. 

Global existence for Yang-Mills fields on $\R^{3+1}$ was shown by
Eardley and Moncrief in a classical result, \cite{EM1} and
\cite{EM2}. Their result was then generalized by Chru\'sciel and
Shatah to general globally hyperbolic curved space-times in
\cite{CS}. Later, the first author wrote in
\cite{G1} a new proof that improves the hypotheses of \cite{CS}.

Our motivation in studying the Yang-Mills equations on the
Schwarzschild geometry is twofold. On the one hand, Yang-Mills fields
are important from a physical point of view and their study on an
important physical space-time like the Schwarzschild metric is therefore
an important problem. On the other hand, the Yang-Mills equations are
linked to the Einstein equations via the Cartan formalism. One
therefore generally hopes to get some insight into the Einstein
equations by studying the Yang-Mills equations. This point becomes
particularly important in the context of the nonlinear stability problem of these space-times. 

The Schwarzschild metric is a  solution of the Einstein vacuum
equations. The Minkowski metric, that describes flat space-time, can be
seen as a special case of the Schwarzschild space-time which itself is
part of the Kerr family of solutions of the Einstein vacuum
equations. The Schwarzschild family describes spherically symmetric
black holes whereas the Kerr family describes rotating black holes, see
\cite{HE} for a description of these space-times.  

The stability of Minkowski space-time was
first proved by Christodoulou and Klainerman in \cite{C-K}. The proof
of Christodoulou and Klainerman was later simplified by
Lindblad-Rodnianski \cite{LR10} using wave coordinates. 
In wave coordinates, the Einstein
equations can be written as a covariant wave equation on the metric,
with a non-linear term depending on the metric, propagating on the
space-time with the metric generated from the solution of this
equation. Very recently Hintz and Vasy proved nonlinear stability of
the De Sitter Kerr space-time, see \cite{HV16}. The De Sitter Kerr
metric is a solution of the Einstein vacuum equations with positive
cosmological constant. The equivalent conjecture for Kerr space-time
has not been solved as of today.  The conjecture of the nonlinear stability of the (De Sitter) Kerr solution has
motivated a lot of work in recent years on proving dispersive
properties for solutions of linear hyperbolic equations on (De Sitter) Schwarzschild and (De
Sitter) Kerr space-time and a lot of progress has been made on this
question. Dispersive properties of linear hyperbolic equations on the De
Sitter Kerr metric have been an essential ingredient in the proof of
nonlinear stability of the De Sitter Kerr metric by Hintz and Vasy.  

Concerning the dispersive
properties of the wave equation, we cite the papers of
Andersson-Blue \cite{AB15}, Dyatlov \cite{Dy11_01}, Dafermos-Rodnianski
\cite{DaRo13_01}, Finster-Kamran-Smoller-Yau \cite{FKSY}, \cite{FKSYER},
Tataru-Tohaneanu \cite{TaTo11}, and Vasy \cite{Va13} as well as references
therein for an overview. The wave equation is of course only a simplification of a linearization of the Einstein
equations around the Kerr solution. Recently Dafermos, Holzegel and
Rodnianski \cite{DHR16}, and Finster, Smoller \cite{FS16}, made important progress in understanding the dispersive
properties of the Teukolsky equation on the Schwarzschild and Kerr metric.

Yet, the free scalar wave equation does not admit stationary solutions
on the exterior of the Schwarzschild black hole, whereas the
Yang-Mills equations admit stationary solutions, which induces new
impediment to the problem as one would need to find a way to exclude
them in the proof of decay. Stationary solutions already appear for the Maxwell equations which can be understood as a linear
version of the Yang-Mills equations (or more precisely as the case
where the Lie group is abelian). The Maxwell equation has been studied by 
Andersson-Blue \cite{AB15_02}, Blue \cite{Bl08}, Ghanem \cite{G2},
Hintz-Vasy \cite{HV15} and
Sterbenz-Tataru \cite{ST15}. In the Maxwell case the stationary
solutions are Coulomb solutions and because of the linearity of the
equation one can get rid of the problem by subtracting a suitable
Coulomb solution. Coulomb solutions also exist for the Yang-Mills
fields but not within our special purely magnetic Ansatz. Other
stationary solutions appear within this Ansatz which are not present
in the Maxwell case. There exists an energy gap between the zero
curvature solution and these stationary
solutions. Thanks to this energy
gap, we can show that for small
enough energy, the solutions satisfying our special Ansatz decay to zero in an appropriate weighted
energy norm. Our results are consistent with what was
observed numerically by Bizo\'n, Rostworowski and Zenginoglu, see
\cite{BRZ}.

Another important difference between the wave equation and the Maxwell
equation is the non-scalar character of the latter. To show decay for
the Maxwell fields, one in general uses the scalar wave equation verified
by the middle components of the fields, which can be decoupled from the extreme components in the abelian case of the Maxwell equations, see for example the work of
Andersson-Blue. However, this separation of
the middle components from the extreme components cannot occur in the non-abelian case of the
Yang-Mills equations. Therefore, the first author wrote in  \cite{G2} a new proof of decay for the
Maxwell fields which does not pass
through the decoupling of the middle-components, associated to any Lie group, without any assumption of spherical symmetry
on the initial data, assuming a certain
Morawetz type estimate for the middle components. Later on, Andersson, B\"ackdahl and
  Blue obtained a Morawetz type estimate for the derivatives of the extreme components of the Maxwell
  fields that doesn't rely on the study of the linear wave equation for the
  middle components, see \cite{ABB16}. Nevertheless, as far as we know, their Morawetz type estimate does not allow one to get decay rates for the Maxwell fields without at least passing through the decoupled scalar wave equation for the middle components. In this paper, we prove a Morawetz type estimate stronger than the one assumed in \cite{G2}, without passing through
the scalar wave equation for the middle components, under some assumptions on the initial
data which eliminate the stationary solutions.

One of the advantages of our Ansatz
is that it reduces the Yang-Mills equations themselves to a
nonlinear scalar wave equation:
\begin{equation*}
\pa_{t}^{2} {W}- \pa_{r^*}^{2} W+ \frac{(1- \frac{2m}{r}) }{r^2} W(W^2-1)=0.
\end{equation*}
The above equation has two obvious stationary solutions
  $W_{\pm}=\pm 1$ which correspond to zero Yang-Mills curvature. However, we note 
  that it doesn't seem to be appropriate to linearize around these stationary solutions, because this would require to control quantities which
  depend only on $W-1$ or on $W+1$, which are neither natural in this
  context (in particular, they are not gauge invariant) nor controlled
  by energy estimates.

The function $P =  \frac{(1- \frac{2m}{r}) }{r^2}  $ that
appears in front of the nonlinearity is exactly the same function that appears in front of $-\Delta_{S^2}$ when one studies the linear scalar wave
equation without the spherical symmetry
assumption.
It is therefore not surprising that the difficulties that
appear in trying to show a Morawetz type estimate are in some sense
similar to the ones linked to trapping for the scalar wave equation,
although no trapping appears here because of the spherical symmetry
assumption. The solution of the problem is however quite different in
this nonlinear setting. The proof of the Morawetz estimate relies on
a nonlinear multiplier, see Section
\ref{proofMorawetz}. Once the Morawetz estimate is established, we adapt the methods of \cite{G2} to the current
  situation. In fact, the paper \cite{G2}
generalizes the arguments of Dafermos and Rodnianski for the free
scalar wave equation, \cite{DR1},  to the Maxwell fields using the
Maxwell equations directly, and thereby, it extends to the nonlinear case of the Yang-Mills fields.

While we need our specific
  Ansatz in order to show the Morawetz estimate, the decay estimates
  can all be understood as corollaries of a general Morawetz estimate. The
  arguments in Sections \ref{sec3} and \ref{Sectionhorizon} are in fact more
  general than what is strictly needed here. This has the
  advantage, however, that a generalization of our result -- for instance, dropping
  the spherical symmetry assumption -- would be reduced to showing
  the Morawetz estimate in such a more general situation. For
details of the calculations of the different tensors and energies, we
refer the reader to \cite{G2}.

\subsection{The exterior of the Schwarzschild black hole} The exterior
Schwarzschild spacetime is given by ${\mathcal M}=\R_t\times
\R_{r>2m}\times S^2$ equipped with the metric 
\beaa
\notag
g &=&  - (1 - \frac{2m}{r})dt^{2} + \frac{1}{ (1 - \frac{2m}{r})} dr^{2} + r^{2} d\th^{2} +  r^{2}\sin^{2} (\th) d\phi^{2} \\
&=& N(-dt^2+d{r^*}^{2})+r^2d\si^2
\eeaa
where
\bea
N &=& (1 - \frac{2m}{r}) \\
r^* &=&  r + 2m\log(r - 2m) - 3m - 2m\log(m)
\eea
and $d\si^2$ is the usual volume element on the sphere. Note that 
\begin{equation*}
\frac{dr^*}{dr}=N^{-1},\quad r^*(3m)=0.
\end{equation*}
The coordinates $t,r, \th, \phi$, are called Boyer-Lindquist coordinates. The
singularity $r=2m$ is a coordinate singularity and can be removed by
changing coordinates, see \cite{HE}. $m$ is the mass of the black hole. We will only
be interested in the region outside the black hole, $r>2m$. If we define,
\beaa
v &=& t + r^{*}  \\
w &=& t - r^{*} 
\eeaa
then, we have,
\beaa
\notag
g &=&  - N dv dw + r^{2}d\sigma^{2} \\
&=&  - \frac{N}{2} dv\otimes dw - \frac{N}{2} dw \otimes dv + r^{2}d\sigma^{2}
\eeaa

Let,
\bea
\hat{\frac{\pa}{\pa w}}  &=&  \frac{1}{N} \frac{\pa}{\pa w} \\
\hat{\frac{\pa}{\pa v}}  &=&   \frac{\pa}{\pa v} \\
\hat{\frac{\pa}{\pa \th}}  &=&  \frac{1}{r} \frac{\pa}{\pa \th} \\
\hat{\frac{\pa}{\pa \phi}}  &=& \frac{1}{r \sin\th }  \frac{\pa}{\pa \phi}
\eea

We will also consider the extended Schwarzschild solution.  It is obtained by making Kruskal's choice of coordinates:
\bea
v^{'} &=& \mbox{exp} (\frac{v}{4m}) \\
w^{'} &=& - \mbox{exp}(-\frac{w}{4m})
\eea
and then define,
\bea
t^{'} &=& \frac{v^{'} + w^{'}}{2} \label{Schwarzschildtime} \\
x^{'} &=& \frac{v^{'} - w^{'}}{2} 
\eea
We get
\bea
g = \frac{16 m^{2}}{r} \exp(\frac{-r}{2m}) ( - (dt^{'})^{2} + (dx^{'})^{2}  ) + r^{2}(t^{'}, x^{'}) d\sigma^{2}. 
\eea
The following figure shows the Kruskal extension of the Schwarzschild
metric:
\begin{center}
\includegraphics[width=8cm]{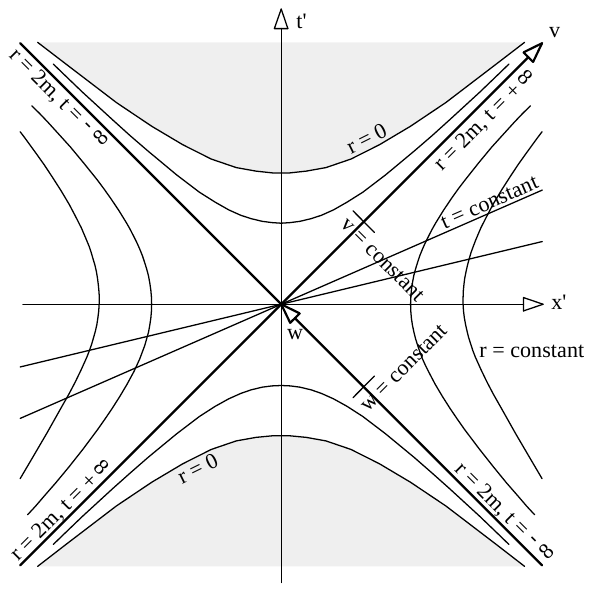}
\end{center}

\subsection{The spherically symmetric $SU(2)$ Yang-Mills equations on the Schwarzschild metric} \label{sphericallysymmetricYM}

Let $G = SU(2)$, the real Lie group of $2 \times 2$ unitary matrices
of determinant 1. The Lie algebra associated to $G$ is $su(2)$, the
antihermitian traceless $2 \times 2$  matrices. Let $\tau_{j}$, $j
\in \{1, 2, 3 \}$, be the following real basis of $su(2)$:
\begin{eqnarray*}
\tau_1=\frac{i}{2}\left(\begin{array}{cc} 0 & 1 \\ 1 &
    0\end{array}\right),\quad
\tau_2=\frac{1}{2}\left(\begin{array}{cc} 0 & -1 \\ 1 &
    0\end{array}\right),\quad
\tau_3=\frac{i}{2}\left(\begin{array}{cc} 1 & 0\\ 0 &
    -1 \end{array}\right). 
\end{eqnarray*}
Note that 
\beaa
[\tau_1,\tau_2]=\tau_3,\quad  [\tau_3,\tau_1]=\tau_2,\quad [\tau_2,\tau_3]=\tau_1.
\eeaa

We are looking for a connection $A$, that is a one form with values in the Lie algebra $su(2)$ associated to the Lie group $SU(2)$, which satisfies the Yang-Mills equations which are:
\bea
\text{\bf D}^{(A)}_{\a} F^{\a\b} \equiv \der_{\a} F^{\a\b} + [A_{\a}, F^{\a\b} ]  = 0, \label{eq:YM}
\eea
where $[.,.]$ is the Lie bracket and $F_{\a\b}$ is the Yang-Mills curvature given by
 \bea
F_{\a\b} = \der_{\a}A_{\b} - \der_{\b}A_{\a} + [A_{\a},A_{\b}], \label{defYMcurvature}
\eea
and where we have used the Einstein raising indices convention with respect to the Schwarzschild metric. We also have the Bianchi identities which are always satisfied in view of the symmetries of the Riemann tensor and the Jacobi identity for the Lie bracket:
\bea
\text{\bf D}^{(A)}_{\a}F_{\mu\nu} + \text{\bf D}^{(A)}_{\mu}F_{\nu\a} + \text{\bf D}^{(A)}_{\nu} F_{\a\mu} = 0. \label{eq:Bianchi}
\eea

The Cauchy problem for the Yang-Mills equations formulates as the following: given a Cauchy hypersurface $\Sigma$ in $M$, and a ${\mathcal G}$-valued one form $A_{\mu}$ on $\Sigma$, and a ${\mathcal G}$-valued one form $E_{\mu}$ on $\Sigma$ satisfying
\begin{eqnarray}
\label{YMconstraintsone}
\left.\begin{array}{rcl} E_{t} &=& 0, \\
\textbf{D}^{(A)}_{\mu}E^{\mu} &=& 0\end{array}\right\} 
\end{eqnarray}
we are looking for a ${\mathcal G}$-valued two form $F_{\mu\nu}$ satisfying the Yang-Mills equations such that once $F_{\mu\nu}$ restricted to $\Sigma$ we have
\bea
F_{\mu t} = E_{\mu}  \label{YMconstraintstwo}
\eea
and such that $F_{\mu\nu}$ corresponds to the curvature derived from the Yang-Mills potential $A_{\mu}$, i.e. given by \eqref{defYMcurvature}. Equations \eqref{YMconstraintsone} are the Yang-Mills constraints equations on the initial data.

Any spherically symmetric Yang-Mills potential can be written in the
following form after applying a gauge transformation, see \cite{FM}, \cite{GuHu} and \cite{W},

\begin{eqnarray}
\label{SPAA}
A &=& [ -W_{1}(t, r) \tau_{1} - W_{2}(t, r) \tau_{2} ] d\th  + [ W_{2}
(t, r) \sin (\th) \tau_{1} - W_{1} (t, r) \sin (\th) \tau_{2}] d\phi\nonumber\\
& + & \cos (\th) \tau_{3} d\phi + A_{0} (t, r) \tau_{3} dt  + A_{1} (t, r)  \tau_{3} dr. 
\end{eqnarray}
where $A_{0} (t, r) $, $A_{1} (t, r) $, $W_{1}(t, r)$, $W_{2}(t, r)$
are arbitrary real functions. We then have 
\begin{eqnarray}
\label{CurvSphSym}
\left.\begin{array}{rcl} F_{\th r}=F_{\mu\nu}(\partial_{\theta})^{\mu}(\partial_r)^{\nu} &=& ( \pa_{r} W_{1}  - W_{2} A_{1}) \tau_{1} +  (\pa_{r} W_{2} + W_{1}  A_{1} ) \tau_{2} \\
F_{\th t}  &=& ( \pa_{t} W_{1} - W_{2} A_{0} ) \tau_{1} + ( \pa_{t} W_{2} +W_{1} A_{0} ) \tau_{2}  \\
F_{\phi r} &=& (- \pa_{r}  W_{2} \sin (\th)  - W_{1} A_{1} \sin (\th)
) \tau_{1}\\
&+& (  \pa_{r} W_{1} \sin (\th) - W_{2}  A_{1}  \sin (\th) ) \tau_{2} \\
F_{\phi t} &=&(- \pa_{t} W_{2} \sin (\th) + A_{0} W_{1} \sin (\th) )
\tau_{1}\\
&+& ( \pa_{t} W_{1} \sin (\th) + A_{0} W_{2} \sin (\th) ) \tau_{2} \\
F_{tr} &=& ( \pa_{t} A_{1} - \pa_{r} A_{0}) \tau_{3} \\
F_{\th\phi} &=&   ( W_{1}^{2} + W_{2}^{2} -1 ) \sin (\th) \tau_{3} \end{array}\right\}
\end{eqnarray}

The Yang-Mills system \eqref{eq:YM} can then be written as
\begin{eqnarray}
\label{YM1}
0 &=&  - \frac{1}{r^{2}} W_{1} [ 1 -  (W_{1}^{2}  + W_{2}^{2} )]   +  N (- \pa^{2}_{r} W_{1}  + \pa_{r} (W_{2} A_{1})) - \frac{1}{ N}  ( - \pa^{2}_{t} W_{1} + \pa_{t} (W_{2} A_{0} ) )\nonumber\\
&&+ \frac{2m}{r^2}  (- \pa_{r} W_{1}  + W_{2} A_{1} )   +  N A_{1}  ( \pa_{r} W_{2} + W_{1}  A_{1} )    - \frac{1}{ N}  A_{0} (\pa_{t} W_{2} +W_{1} A_{0} ) \\
\label{YM2}
0 &=& - \frac{1}{r^{2}}  W_{2}  [ 1 - (W_{1}^{2}  + W_{2}^{2}) ] -  N  ( \pa^{2}_{r} W_{2} +\pa_{r} ( W_{1}  A_{1}) )   + \frac{1}{ N}  (   \pa^{2}_{t} W_{2} + \pa_{t} ( W_{1} A_{0}))\nonumber  \\
&& +  \frac{2m}{r^2} ( - \pa_{r} W_{2} - W_{1}  A_{1} )   + N
A_{1} ( - \pa_{r} W_{1}  + W_{2} A_{1})\nonumber\\
&-& \frac{1}{ N}  A_{0}  ( -\pa_{t} W_{1} + W_{2} A_{0})\\
\label{YM3}
0 &=& -  \frac{2}{r^{2}} W_{1}  (\pa_{r} W_{2} + W_{1}  A_{1} )  +  \frac{2}{r^{2}} W_{2} ( \pa_{r} W_{1}  - W_{2} A_{1})
\nonumber\\
& -& \frac{1}{ N}  \pa_{t}  ( \pa_{t} A_{1} - \pa_{r} A_{0}) \\
\label{YM4}
0 &=&  - \frac{2}{r^{2}} W_{1}    ( \pa_{t} W_{2} )  +
\frac{2}{r^{2}}  W_{2}  ( \pa_{t} W_{1} )  - \frac{2N}{r}  ( \pa_{t} A_{1} - \pa_{r} A_{0}) \nonumber\\
&-&  N \pa_{r} ( \pa_{t} A_{1} - \pa_{r} A_{0})   
\end{eqnarray}
\subsection{The initial data} \label{AnsatzforinitialdataYM}
We look at initial data prescribed on $t=t_{0}$ where there exists a
gauge transformation such that once applied on the initial data, the
potential $A$ can be written for some $c\in \R$ in this gauge as
\begin{equation}
\label{Ansatz}
\left.\begin{array}{rcl}
A_{t} (t=t_{0})  &=& 0, \\
A_{r} (t=t_{0})  &=& 0, \\
A_{\th} (t=t_{0})  &=& -W_{1}(t_{0}, r)(\tau_{1} + c  \tau_{2}),  \\
A_{\phi} (t=t_{0})  &=& W_{1}(t_{0}, r)  (c\sin (\th) \tau_{1} - \sin
(\th) \tau_{2})  + \cos (\th) \tau_{3}, \end{array}\right\} 
\end{equation}
and, we are given in this gauge the following one form $E_{\mu}$ on $t=t_{0}$:
\begin{equation}
\label{AnsatzE}\left.\begin{array}{rcl}
E_{\th} (t=t_{0}) &=& F_{\th t} (t=t_{0})   = ( \pa_{t} W_{1}  ) (\tau_{1} + c\tau_{2}),  \\
E_{\phi} (t=t_{0}) &=& F_{\phi t} (t=t_{0})   =  ( \pa_{t} W_{1} )  (- c \sin (\th)  \tau_{1} + \sin (\th)   \tau_{2}) \\
E_{r} (t=t_{0}) &=& F_{rt} (t=t_{0}) =  ( \pa_{t} A_{1} - \pa_{r} A_{0}) \tau_{3} = 0, \\
E_{t} (t=t_{0}) &=& F_{tt} (t=t_{0}) = 0. \end{array}\right\}
\end{equation}
Notice that with this Ansatz the constraint equations \eqref{YMconstraintsone} are
automatically fulfilled
\begin{eqnarray}
\label{constraintintheAnsatz}
\notag
\lefteqn{( { \text{\bf D}^{(A)}}^{\th}  E_{\th} + {\text{\bf D}^{(A)}}^{\phi}  E_{\phi} + {\text{\bf D}^{(A)}}^{r}  E_{r} ) (t=t_{0})}  \\
&=&  c\left(- \frac{2}{r^{2}} W_{1}    ( \pa_{t} W_{1}  )(t_{0}, r)
  \tau_{3} + \frac{2}{r^{2}}  W_{1}  ( \pa_{t} W_{1}  ) (t_{0}, r)
  \tau_{3}\right)=0.  
\end{eqnarray}
\begin{remark}
The principal restriction is $A_t(t=t_0)=A_r(t=t_0)=0$. If we want this to be
conserved, then by restricting \eqref{YM3} to $t=t_0$ we obtain:
\begin{equation}
\label{AdCYM}
W_1\partial_rW_2=W_2\partial_rW_1.
\end{equation}
Suppose that $W_1$ and $W_2$ have no zeros. Then \eqref{AdCYM} gives 
\begin{equation*}
\partial_r\ln W_2=\partial_r\ln W_1
\end{equation*}
and thus $W_2=cW_1$ for some $c\in \R$. 
\end{remark}
Now suppose that $W$ is solution of 
\bea
\ddot{W}-W''+ PW(W^2-1)=0, \label{TranslationofYang-MillsonW}
\eea
where 
\begin{equation*}
\dot{} \equiv \pa_{t}\quad\mbox{and}\quad ' \equiv \pa_{r^*} 
\end{equation*}
and where
\bea
P &\equiv& \frac{N}{r^2}  \label{defP}.
\eea
Then $W_1=\frac{1}{\sqrt{1+c^2}}W,\, W_2=\frac{c}{\sqrt{1+c^2}}W,\, A_0=A_1=0$ are
solutions of \eqref{YM1} to \eqref{YM4}. By the uniqueness of the
solutions of the Yang-Mills equation $F$ defined by \eqref{CurvSphSym}
is the solution of the Yang-Mills equation. For our special Ansatz the
analysis of the Yang-Mills equation therefore reduces to the analysis
of \eqref{TranslationofYang-MillsonW} and this corresponds to the case
$c=0$. As all the analysis reduces to this case we will suppose in the
following $c=0$ and put $W:=W_1$. With this special Ansatz the data
$A_{\mu},\, E_{\mu}$ is equivalent to the data $W (t=t_{0}),\, \partial_tW (t=t_{0})$ and
the equation we have to study is
\eqref{TranslationofYang-MillsonW}.

\subsection{The fundamental scalar wave equation}
We are looking at solutions such that once the above mentioned gauge
transformation is applied the potential $A$ takes the form
\eqref{Ansatz} with $c=0$. We recall here the
  expressions of the components of $A$ and $F$ with our choices:
\beaa
\left.\begin{array}{rcl}A_{t} &=& 0,\\
A_{r^*} &=& 0,  \\
A_{\th} &=& -W  \tau_{1}, \\
A_{\phi} &=&  -  W \sin (\th) \tau_{2}  + \cos (\th) \tau_{3},\end{array}\right\} 
\eeaa
and
\beaa
\left.\begin{array}{rcl} F_{\th r^*} &=& W' \tau_{1},\\
F_{\th t}  &=& \dot{W} \tau_{1}, \\
F_{\phi r^*} &=& W' \sin (\th) \tau_{2}, \\
F_{\phi t} &=&\dot{W} \sin (\th) \tau_{2}, \\
F_{tr^*} &=& 0,\\
F_{\th\phi} &=&   ( W^{2} -1 ) \sin (\th) \tau_{3}. \end{array}\right\}
\eeaa
The principal object of study is now the scalar wave equation
\begin{equation}
\label{YMSW}
\left.\begin{array}{rcl} \ddot{W}-W''+PW(W^2-1)&=&0,\\
W(0)&=&W_0,\\
\partial_t W(0)&=&W_1.\end{array}\right\} 
\end{equation}
It is easy to check that the following energy is conserved, see also \cite{G2},
\begin{equation*}
\cE(W,\dot{W})=\int \dot{W}^2+(W')^2+\frac{P}{2}(W^2-1)^2 dr^*.
\end{equation*}

We note by $\dot{H}^k=\dot{H}^k(\R, dr^*)$ and
  $H^k=H^k(\R,dr^*)$, the homogeneous and inhomogeneous Sobolev spaces
of order $k$, respectively. 

\begin{definition}
\begin{enumerate}
\item We define the spaces $L^4_P$, resp. $L^2_P$, as the completion of
$C_0^{\infty}(\R)$ for the norm
\bea
\Vert v\Vert_{L^4_P}^4:=\int P\vert v\vert^4 dr^*\quad
\mbox{resp.}\quad \Vert v\Vert_{L^2_P}^2:=\int P \vert v\vert^2 dr^*. 
\eea
\item
We also define for $1\le k\le 2$ the space $\cH^k$ as the completion
of $C_0^{\infty}(\R)$ for the norm 
\bea
\Vert u\Vert^2_{\cH^k}=\Vert u\Vert_{\dot{H}^k}^2+\Vert
u\Vert_{L^4_P}^2. 
\eea
\end{enumerate}
\end{definition}
We note that $\cH^k$ is a Banach space which contains all constant functions.
\begin{theorem}
\label{ThGEYM}
Let $(W_0,W_1)\in \cH^2\times H^1$. Then there exists a unique strong solution of
\eqref{YMSW} with
\begin{eqnarray*}
W&\in&C^1([0,\infty);\cH^1)\cap
C([0,\infty);\cH^2),\\
\partial_tW&\in& C^1([0,\infty);L^2)\cap C([0,\infty);H^1),\\
\sqrt{P}(W^2-1)&\in&C^1([0,\infty);L^2)\cap C([0,\infty);H^1). 
\end{eqnarray*}
\end{theorem}
We will prove the theorem in Appendix \ref{AppendixB}. In this
appendix, we will also show that $\cH^1\times L^2$ is exactly the space
of finite energy solutions. Let $\Om
\equiv  \{ (t', x') \in \R^2 \vert x' > \vert t' \vert \}$. We can
reformulate the above theorem in the following way
\begin{corollary}
\label{Cor1}
We suppose that the initial data for the Yang-Mills equations is given
after suitable gauge transformation by
\begin{eqnarray*}
\left.\begin{array}{rcl} A_t(0)&=&A_r(0)=0,\\
    A_{\theta}(0)&=&-W_0\tau_1,\\
 A_{\phi}(0)&=&-W_0\sin
    \theta\tau_2+\cos\theta\tau_3,\\
E_{\theta}(0)&=&W_1\tau_1,\\
E_{\phi}(0)&=&W_1\sin\theta\tau_2,\\
E_r(0)&=&E_t(0)=0\end{array}\right\} 
\end{eqnarray*}
with $(W_0,W_1)\in \cH^2\times H^1$. Then, the Yang-Mills equation \eqref{eq:YM} admits a unique
solution $F$ with
\begin{eqnarray*}
F_{\theta r^*},\, \frac{1}{\sin\theta}F_{\phi r^*},\, F_{\theta t},\,
\frac{1}{\sin\theta}F_{\phi
  t},\sqrt{P}\frac{1}{\sin\theta}F_{\theta\phi}&\in&
C^1([0,\infty);L^2)\cap C([0,\infty);H^1),\\
\forall X, Y \in \{\hat{\frac{\pa}{\pa w}}, \hat{\frac{\pa}{\pa v}},
\hat{\frac{\pa}{\pa \th}}, \hat{\frac{\pa}{\pa \phi}}  \},\quad F_{\mu\nu}
X^{\mu}Y^{\nu} (t', x')&\in& H^{1}_{loc}(\Om ),\\
\forall X, Y \in \{\hat{\frac{\pa}{\pa w}}, \hat{\frac{\pa}{\pa v}},
\hat{\frac{\pa}{\pa \th}}, \hat{\frac{\pa}{\pa \phi}}  \},\quad \lim_{r^* \to \infty} F_{\mu\nu} X^{\mu}Y^{\nu}(t, r^*) &=& 0.
\end{eqnarray*}
\end{corollary}
\proof 

We only have to check the last two statements which follow from
\begin{eqnarray*}
F_{\hat{v}\hat{\theta}}&=&-\frac{2}{r}\partial_vW\tau_1,\,
F_{\hat{w}\hat{\theta}}=-\frac{2}{Nr}\partial_wW\tau_1,\,
F_{\hat{v}\hat{\phi}}=-\frac{2}{r}\partial_vW\tau_2,\\
F_{\hat{w}\hat{\phi}}&=&-\frac{2}{Nr}\partial_wW\tau_2,\,
F_{\hat{w}\hat{v}}=0,\,
F_{\hat{\theta}\hat{\phi}}=\frac{W^2-1}{r^2}\tau_3
\end{eqnarray*}
and the Sobolev embedding $H^1(\R)\subset C_b(\R)$, where $C_b(\R)$ is
the set of bounded continuous functions equipped with the $L^{\infty}$
norm. 
\qed
\begin{remark}
\begin{enumerate}
\item Note that our functional setting doesn't impose any specific asymptotic
    behaviour on the solutions.
\item Strictly speaking, the initial data are functions on $\R\times
  S^2$, and $W_0$ has to be in $\cH^2\otimes
  L^2(S^2)$. $F_{\theta r^*}$ takes by then values in $H^1\otimes L^2(S^2)$ etc. We will, in the
  following, quite often ignore the $L^2(S^2)$ factor which is constant.
\item Global existence for Yang-Mills fields is of course known in a
  more general context, see e.g. \cite{G1}. Nevertheless, we prefer to
  give here a theorem for our special Ansatz. The advantage is that
  the formulation of this theorem and of its proof, is particularly simpler
  in this special case.
\item The regularity results imply that we can apply the divergence
  theorem in the exterior of the black hole. Under
    some regularity assumptions on the initial data, the
    result of \cite{G1} shows that the traces of suitable normalized
    components exist up to the
    horizon. The existence of the traces at the horizon at the
    regularity level of the above Corollary, follows from our estimates
    in Section \ref{Sectionhorizon}. In particular, we can then apply the
    divergence theorem up to the horizon.
\end{enumerate}
\end{remark}

\subsection{Stationary solutions}
Note that $W_{\pm}=\pm 1$ and $W_{\infty}=0$ are obvious stationary
solutions of \eqref{TranslationofYang-MillsonW}. The solutions $W_{\pm}$ correspond
to zero Yang-Mills curvature. Other stationary solutions are given by the
following theorem, which is implicit in the paper
\cite{BRZ} of P. Bizo\'n, A. Rostworowski and A. Zenginoglu.
\begin{theorem}
\label{thstat}
There exist a decreasing sequence $\{a_n\}_{n\in \N^{\ge1}},\, 0<...<
a_n< a_{n-1}<...<a_1=\frac{1+\sqrt{3}}{3\sqrt{3}+5}$, and  $W_n$ smooth stationary solutions of \eqref{TranslationofYang-MillsonW},
with 
\begin{equation*}
-1\le W_n\le 1,\quad \lim_{x\rightarrow -\infty}W_n(x)=a_n,\quad \lim_{x\rightarrow
  \infty}W_n(x)=(-1)^n. 
\end{equation*}
For each $n \in \N^{\ge1}$, the solution $W_n$ has exactly $n$ zeros. 
\end{theorem}
It is implicitly stated in \cite{BRZ}, that there is an energy gap between
the $W_{\pm}=\pm 1$ solutions and the next stationary solutions, a
statement which is confirmed by our analysis, where we also show that the zero Yang-Mills curvature solution is stable under a small perturbation. In \cite{HaHu}, H\"afner and Huneau show that all the
  solutions constructed in Theorem \ref{thstat} are nonlinearly
  instable. The paper \cite{HaHu} contains also a detailed proof of Theorem \ref{thstat}.
\subsection{Energy estimates}

Let $<\; ,\; >$ be an Ad-invariant scalar product on the Lie algebra $su(2)$, i.e. for any $su(2)$-valued tensors $A$, $B$, and $C$, we have
\beaa
<[A, B], C> = <A, [B, C]>
\eeaa
where the Lie bracket $[.,.]$ can be defined as corresponding to commutation of matrices.

Let $F$ be the Yang-Mills curvature solution to the Yang-Mills equations. We consider the Yang-Mills energy-momentum tensor, see \cite{G1},
\beaa
T_{\mu\nu}^{}(F) = <F_{\mu\b}, {F_{\nu}}^{\b}> - \frac{1}{4} \g_{\mu\nu} < F_{\a\b}, F^{\a\b}>
\eeaa
We have (see \cite[page 17]{G2}):
\begin{equation}
\label{divergencefree}
\nabla^{\nu}T_{\mu\nu}=0.
\end{equation}
Considering now a vector field $X^{\nu}$ we let $$J_{\mu}(X) =
X^{\nu}T_{\mu\nu} = T_{\mu X}.$$
Using \eqref{divergencefree} we obtain 
\begin{equation*}
\nabla^{\mu}J_{\mu}(X)=\pi^{\mu\nu}(X)T_{\mu\nu},
\end{equation*}
where $\pi^{\mu\nu}$ is the deformation tensor defined by 
\bea
\pi^{\mu\nu} (X) = \frac{1}{2} ( \der^{\mu}X^{\nu} +
\der^{\nu}X^{\mu}). 
\eea
 
Applying the divergence theorem on $ J_{\mu}(X)$ in the region $B$ bounded to the past by $\Sigma_{t_{1}}$ and to the future by $\Sigma_{t_{2}}$, we obtain: \
\bea
\notag
 \int_{B}  \pi^{\mu\nu}(X) T_{\mu\nu} dV_{B}  &=& \int_{\Sigma_{t_{1}}} J_{\mu}(X) n^{\mu} dV_{\Sigma_{t_{1}}} -  \int_{\Sigma_{t_{2}}} J_{\mu}(X) n^{\mu} dV_{\Sigma_{t_{2}}}   \\
 &=& E_{F}^{(X)} (t_{1}) -  E_{F}^{(X)} (t_{2})   \label{conservationlawdivergncetheorem} 
\eea
where $n^{\mu}$ are the unit normal to the hypersurfaces $\Sigma_{t}$, and
\bea
E_{F}^{(X)} (t) =  \int_{\Sigma_{t}} J_{\mu}(X) n^{\mu} dV_{\Sigma_{t}}. 
\eea
We refer to \eqref{conservationlawdivergncetheorem} as the energy identity associated to the vector field $X$ as multiplier.\
If $X$ is Killing, then the deformation tensor $\pi^{\mu\nu}(X)$ is
zero and \eqref{conservationlawdivergncetheorem} gives a conserved
energy. In the Schwarzschild case, the vector field $\partial_t$ is
Killing and we obtain the conserved energy:  
\begin{eqnarray*}
E_{F}^{  (\frac{\pa}{\pa t} )}&=& \int_{r^{*} = - \infty
}^{r^{*} = \infty }  \int_{\phi = 0}^{ 2\pi} \int_{\th= 0 }^{\pi }   2
[   N^{2} ( | F_{\hat{w}\hat{\th}} |^{2} +  | F_{\hat{w}\hat{\phi}}
|^{2} ) +   | F_{\hat{v}\hat{\th}} |^{2} +  | F_{\hat{v}\hat{\phi}}
|^{2} \\
&+&   N (  |F_{\hat{v}\hat{w}}|^{2}    + \frac{1}{4} | F_{\hat{\phi}\hat{\th}}|^{2}  )  ] . r^{2}  \sin(\th) d\th d\phi dr^{*}.
\end{eqnarray*}
where $|\: .\;|$ is the norm associated to the Ad-invariant scalar product $<\; ,\;>$. Note that in our special case the energy simply reads
\begin{equation*}
E_{F}^{  (\frac{\pa}{\pa t} )}=\int_{\R}\int_{S^2}
\dot{W}^2+(W')^2+\frac{P}{2}(W^2-1)^2 dr^*d\si^2. 
\end{equation*}
We will often apply the divergence theorem to vector fields of the
form 
\beaa
X = X^{w}(v, w) \frac{\pa}{\pa w} + X^{v} (v, w) \frac{\pa}{\pa v}
\eeaa
Then we have, see \cite[page 113]{G2},
\beaa
\notag
&& \pi^{\a\b}(X)T_{\a\b}(F) \\
\notag
&=& (| F_{\hat{w}\hat{\th}} |^{2} + | F_{\hat{w}\hat{\phi}} |^{2})( -2N \pa_{v}X^{w})  + ( | F_{\hat{v}\hat{\th}} |^{2} + | F_{\hat{v}\hat{\phi}} |^{2} ) (\frac{-2}{N} \pa_{w}X^{v})\\
\notag
&& + (  |F_{\hat{v}\hat{w}}|^{2}   +  \frac{1}{4} | F_{\hat{\phi}\hat{\th}}|^{2}  ) ( -2 [    \pa_{v} X^{v}  + \pa_{w} X^{w}   +      \frac{(3\mu-2)}{2r} ( X^{v}  -   X^{w} ) ]  ) 
\eeaa
where
\bea
\mu &\equiv& \frac{2m}{r}. 
\eea

\subsection{Main result}
Because of the existence of stationary
solutions with finite conserved energy, other than the zero curvature
solution, there can't be any Morawetz estimate that holds for all
finite energy solutions. Nevertheless, we obtain a Morawetz estimate
that holds for initial data with small enough energy. More precisely,
we will prove the following theorem :
\begin{theorem}
\label{th2}
There exists $\epsilon>0$ with the following property. For all
solutions $F$ of \eqref{eq:YM} with initial data as in
  Corollary \ref{Cor1} and such that $E_F^{(\frac{\pa}{\pa t})} (t=t_{0}) <\epsilon$,  we have for all $t$,
\bea
\notag
 &&\int_{t=t_{0}}^{t} \int_{r^*=-\infty}^{\infty} \int_{\S^2} [   P
 N^{2} ( | F_{\hat{w}\hat{\th}} |^{2} +  | F_{\hat{w}\hat{\phi}} |^{2}
 ) +   \frac{N}{r} (  |F_{\hat{v}\hat{w}}|^{2}    +  |
 F_{\hat{\phi}\hat{\th}}|^{2}  )   +  P ( | F_{\hat{v}\hat{\th}} |^{2}
 +   | F_{\hat{v}\hat{\phi}} |^{2}  )  ]  r^{2}  d\si^2 dr^{*}dt  \\
\notag
&=&\int_{t=t_{0}}^{t} \int_{r^*=-\infty}^{\infty} \int_{\S^2}
  P\dot{W}^2+P(W')^2+\frac{P}{2r}(W^2-1)^2 dr^*d\si^2 dt\\
&\lesssim & E_F^{(\frac{\pa}{\pa t})} (t=t_{0})   \label{Morawetz-estimate}
 \eea
and, the local energy decays as,
\bea
E_{F}^{(\frac{\pa}{\pa t})}  ( r^*_1 \leq r^* \leq r^*_2 ) &\leq& C(r_{1}^*, r_{2}^*)  \frac{( E_{F}^{(\frac{\pa}{\pa t})} (t_{0}) + E_{F}^{(K)}(t_{0})  ) }{ t}. 
\eea
Furthermore, under some additional regularity
  assumptions on the initial data, for all $ r \geq 2m $ (including the event horizon), we have,
\bea
|F_{\hat{\th}\hat{\phi} } | +  |F_{\hat{v} \hat{w}}  | =  | \frac{W^2-1}{r^2} |  (v, w, \th, \phi)      \leq  \frac{ C }{ ( \max\{1, v \} )^{\frac{1}{2}} }.
\eea

More precisely, let $R>2m$.  Then we have in the region  $ r \geq R >
2m$ (away from the event horizon),
\bea
|F_{\hat{\th}\hat{\phi} } | +  |F_{\hat{v} \hat{w}}  | =   | \frac{W^2-1}{r^2} |  (v, w, \th, \phi)     &\lesssim& \frac{  E_{1} }{ (1 + |v|)^{\frac{1}{2}} }, 
\eea
\bea
|F_{\hat{\th}\hat{\phi} } | +  |F_{\hat{v} \hat{w}}  | =   | \frac{W^2-1}{r^2} |  (v, w, \th, \phi)    &\lesssim& \frac{  E_{1} }{ (1 + |w|)^{\frac{1}{2}} }, 
\eea
and in the region $2m \leq r \leq R $ (near the event horizon), 
\bea
|F_{\hat{\th}\hat{\phi} } | +  |F_{\hat{v} \hat{w}}  | \les | W^2-1 |  (v, w, \th, \phi)   \lesssim  \frac{ E_{2} }{ (\max\{1, v \} )^{\frac{1}{2}}},
\eea
with
\begin{eqnarray*}
E_{1} &= & [   E_{   F }^{(\frac{\pa}{\pa t}) } (t=t_{0})  +    E_{  F }^{( K ) } (t=t_{0}) ]^{\frac{1}{2}},\\
E_{2} &= & [  (E_{   F }^{(\frac{\pa}{\pa t}) } (t=t_{0})  +    E_{
  F }^{( K ) } (t=t_{0})  + 1)^2  +  E_{F }^{ \# (\frac{\pa}{\pa t} )}
( t= t_{0} ) ]^{\frac{1}{2}},
\end{eqnarray*}

where,
\begin{eqnarray*}
 E_{F}^{  (\frac{\pa}{\pa t} )} ( t= t_{0} )  &=& \int_{r^{*} = - \infty }^{r^{*} = \infty }  \int_{\phi = 0}^{ 2\pi} \int_{\th= 0 }^{\pi }   2 [   N^{2} ( | F_{\hat{w}\hat{\th}} |^{2} +  | F_{\hat{w}\hat{\phi}} |^{2} ) +   | F_{\hat{v}\hat{\th}} |^{2} +  | F_{\hat{v}\hat{\phi}} |^{2}  \\
&& +   N (  |F_{\hat{v}\hat{w}}|^{2}    + \frac{1}{4} | F_{\hat{\phi}\hat{\th}}|^{2}  )  ] . r^{2}  \sin(\th) d\th d\phi dr^{*},\\
E^{(K)}_{F} (t_{i}) &=& \int_{r^{*} = - \infty}^{r^{*}= \infty}    \int_{\phi = 0}^{ 2\pi}  \int_{\th= 0 }^{\pi }   (   w^{2} N^{2} [ | F_{\hat{w}\hat{\th} } |^{2} +  | F_{\hat{w} \hat{\phi} } |^{2} ]  +  v^{2}  [  | F_{\hat{v} \hat{\th} } |^{2} +  | F_{\hat{v} \hat{\phi} } |^{2} ]  \\
&& \quad \quad \quad \quad \quad +  (\om^{2} + v^{2} )  N [  |F_{\hat{v} \hat{w} }|^{2}   +   \frac{1}{4} | F_{\hat{\phi}\hat{\th} }|^{2}]  )   r^{2}  \sin(\th)d\th d\phi dr^{*},\\
 E_{F}^{ \# (\frac{\pa}{\pa t} )} ( t= t_{0} ) &=& \int_{r^{*} = - \infty }^{r^{*} = \infty } \int_{\phi = 0}^{ 2\pi}  \int_{\th= 0 }^{\pi }    [  N (  | F_{\hat{w}\hat{\th}} |^{2} +  | F_{\hat{w}\hat{\phi}} |^{2}   ) +(  | F_{\hat{v}\hat{\th}} |^{2} +  | F_{\hat{v}\hat{\phi}} |^{2} ) \\
&& + (  |F_{\hat{v}\hat{w}}|^{2}   +  \frac{1}{4} | F_{\hat{\phi}\hat{\th}}|^{2} )  ] . r^{2} \sin(\th)d\th d\phi  dr^{*} ( t = t_{0} ).
\end{eqnarray*}
\end{theorem}

\textbf{Acknowledgments.} The first author would like to thank Universit\'e Joseph Fourier - Grenoble I, for providing funding to support this research, and its Department of Mathematics for the kind hospitality, and thanks very warmly all its staff for their impressive efficiency in doing their work. The second author acknowledges support from the ANR funding ANR-12-BS01-012-01. 
\section{Proof of the Morawetz estimate \eqref{Morawetz-estimate}}
\label{proofMorawetz}
Recall that $\mu = \frac{2m}{r}$,\quad $N=1-\mu$,\quad  $P=\frac{N}{r^2}$. We have 
\begin{equation*}
P'=PV,\quad\mbox{where}\quad V=\frac{3\mu-2}{r}. 
\end{equation*}
We compute
\begin{equation*}
V'=N\partial_r\left(\frac{3\mu-2}{r}\right)=\frac{2N}{r^2}(1-3\mu)=2P(1-3\mu).
\end{equation*}
The following proposition will be useful 
\begin{proposition}
\label{prop1}
\begin{equation*}
\Vert \sqrt{P}(W^2-1)\Vert_{L^{\infty}(\R)} \lesssim  \sqrt{E_F^{(\frac{\pa}{\pa t})}} + E_F^{(\frac{\pa}{\pa t})} 
\end{equation*}
\end{proposition}
\proof 

We use the Sobolev embedding $H^1(\R)\subset L^{\infty}(\R)$:
\begin{eqnarray*}
\Vert \sqrt{P}(W^2-1)\Vert^2_{L^{\infty}(\R)}&\lesssim& \int
P(W^2-1)^2+\int ((\sqrt{P})')^2(W^2-1)^2+\int P ((W^2-1)')^2\\
&=&E_F^{(\frac{\pa}{\pa t})} +\frac{1}{4}\int PV^2(W^2-1)^2+\int 4
PW^2(W')^2\\
&\lesssim& (1+\Vert
\sqrt{P}(W^2-1)\Vert_{L^{\infty}(\R)})  E_F^{(\frac{\pa}{\pa t})}
\end{eqnarray*} 
The proposition follows. 
\qed

\subsection{A first multiplier}
Let
\beaa
G = f(r^{*}) \frac{\pa}{\pa r^*} &=& -f(r^{*}) \frac{\pa}{\pa w} + f(r^{*}) \frac{\pa}{\pa v}
\eeaa
where
\bea
f=\frac{1}{3m}-\frac{1}{r}
\eea
We have 
\bea
f'=\frac{N}{r^2}=P.
\eea
We obtain (see \cite[(49)]{G2}):
\bea
\notag
 T^{\a\b}(F)\pi_{\a\b}(G) dVol &=&  T^{\a\b}(F)\pi_{\a\b}(G) N r^2 dr^*d\si^2 dt \\
\notag
 &=& [ ( N^2 | F_{\hat{w}\hat{\th}} |^{2} + N^2 | F_{\hat{w}\hat{\phi}} |^{2} +  | F_{\hat{v}\hat{\th}} |^{2} +  | F_{\hat{v}\hat{\phi}} |^{2} )  f'  \\
\notag
&& -2 N (  | F_{\hat{v}\hat{w}}|^{2}  +  \frac{1}{4 } | F_{\hat{\phi}\hat{\th}}|^{2} ) (  f' +  \frac{f}{r}(3\mu -2)  ) ]  r^2 dr^*d\si^2 dt  \label{contracteddeformationforG} \\
&=& [  \frac{1}{2} P \dot{W}^2 + \frac{1}{2}  P(W')^2-\frac{1}{4} P(Vf+f')(W^2-1)^2 ]  dr^*d\si^2 dt \nonumber
\eea
and (see \cite[(52)]{G2}):
\bea
\notag
E^{(G)}_{F} (t)&=& \int_{r^{*} = - \infty}^{r^{*}= \infty} \int_{\S^{2}} - f   [    \frac{1}{r^{2}} <F_{t\th}, {F_{r^{*}\th}}^{}> +  \frac{1}{r^{2} \sin^{2} \th}  <F_{t\phi}, {F_{r^{*}\phi}}^{}>          ]   r^{2}  d\si^{2} dr^{*}  \\ 
\notag
&=& \int_{r^{*} = - \infty}^{r^{*}= \infty} \int_{\S^{2}} - f N   [    \frac{1}{r^{2}} <F_{t\th}, {F_{r\th}}^{}> +  \frac{1}{r^{2} \sin^{2} \th}  <F_{t\phi}, {F_{r\phi}}^{}>          ]   r^{2}  d\si^{2} dr^{*}  \\ 
\notag
&=& \int_{r^{*} = - \infty}^{r^{*}= \infty} \int_{\S^{2}} -  f P   ( \pa_{t} W ) ( \pa_{r} W)    r^{2}  d\si^{2} dr^{*}  \\
&=& \int_{r^{*} = - \infty}^{r^{*}= \infty} \int_{\S^{2}} -  f    \dot{W} W'   d\si^{2} dr^{*}   \label{definitionoftheboundarytermgeneratedfromG} 
\eea
Applying the divergence theorem in the region $t_0\le t'\le t$ and
differentiating with respect to $t$ we obtain: 
\begin{equation*}
\frac{d}{dt}\int_{\R}\dot{W}fW'dr^*-\frac{1}{4}\int_{\R}P(Vf+P)(W^2-1)^2dr^*+\frac{1}{2}\int_{\R} P(W')^2dr^*+\frac{1}{2}\int_{\R}P\dot{W}^2dr^*=0.
\end{equation*} 
Note that 
\begin{equation*}
-Vf=2\frac{(r-3m)^2}{3mr^3}\ge 0.
\end{equation*}

\subsection{A second multiplier}
Let $h$ be a smooth function with compact support. We multiply the Yang-Mills equation \eqref{TranslationofYang-MillsonW} by
$hW(W^2-1)$ and integrate by parts. We obtain
\beaa
0 &=&  \int_{\R} \ddot{W}hW(W^2-1) dr^* - \int_{\R} W''hW(W^2-1) dr^* + \int_{\R}  PhW^2(W^2-1)^2 dr^*.
\eeaa
\begin{enumerate}
\item First term. We have
\beaa
\frac{d}{dt}  \int_{\R}  \dot{W}hW(W^2-1) dr^* &=& \int_{\R} \ddot{W}hW(W^2-1) dr^*+  \int_{\R} 
\dot{W}^2h(W^2-1) dr^* \\
&& +2  \int_{\R}  \dot{W}^2hW^2 dr^*.
\eeaa
Thus 
\beaa
  \int_{\R}  \ddot{W}hW(W^2-1) dr^* &=& \frac{d}{dt}  \int_{\R}  \dot{W}hW(W^2-1) dr^* -  \int_{\R} 
h\dot{W}^2(W^2-1)dr^* \\
&&-2  \int_{\R}  h\dot{W}^2W^2 dr^*.
\eeaa
\item Second term. We have 
\beaa
-  \int_{\R} W''hW(W^2-1) dr^* &=&  \int_{\R}  W'h'W(W^2-1) dr^* +  \int_{\R} (W')^2h(W^2-1) dr^* \\
&& +2 \int_{\R}  h(W')^2W^2 dr^*.
\eeaa
 We have 
\beaa
  \int_{\R} W'h'W(W^2-1) dr^* &=&\frac{1}{4}  \int_{\R}  h'\frac{d}{dr^*}(W^2-1)^2 dr^*  =-\frac{1}{4}  \int_{\R}  h''(W^2-1)^2 dr^*.
\eeaa
Thus 
\beaa
-  \int_{\R} W''hW(W^2-1) dr^* &=& -\frac{1}{4}  \int_{\R}  h''(W^2-1)^2 dr^*   \\
&& +  \int_{\R}  (W')^2h(W^2-1) dr^* +2  \int_{\R} 
h(W')^2W^2 dr^*.
\eeaa
\item Third term. We have 
\begin{equation*}
 \int_{\R}  P hW^2(W^2-1)^2 dr^* =  \int_{\R}  P h(W^2-1)^2 dr^* +  \int_{\R}  P h(W^2-1)^3 dr^* .
\end{equation*}
\item Summarizing we have 

\bea
\notag
&& \frac{d}{dt} \int_{\R}  \dot{W}hW(W^2-1) dr^* -\int_{\R} 
  h\dot{W}^2(W^2-1) dr^* -2 \int_{\R}  h \dot{W}^2 W^2 dr^* \nonumber  \\
 \notag
&& -\frac{1}{4}\int_{\R} h''(W^2-1)^2 dr^*  + \int_{\R}  (W')^2h(W^2-1) dr^* + 2 \int_{\R}  h (W')^2W^2 dr^* \nonumber\\
\notag
&& + \int_{\R} P h(W^2-1)^2 dr^*+\int_{\R}  P h(W^2-1)^3 dr^* \\
\notag
&& + \frac{1}{2}\int_{\R}  P \dot{W}^2 dr^* +\frac{1}{2}\int_{\R}
P (W')^2 dr^* -\frac{1}{4}\int_{\R}  P (Vf+f')(W^2-1)^2 dr^* \\
&& +\frac{d}{dt} \int_{\R}  \dot{W}fW' dr^*=0  \label{summofallterms}
\eea

\end{enumerate}
\subsection{Choice of $h$}
Let $\chi(r^*) \in C_0^{\infty}((-2,2)),\, \chi\ge 0,\, \chi(r^*) \equiv 1$ on
$[-1,1]$. For a smooth function $\psi(r^*)$ we put
$\psi_a(r^*)=\psi\left(\frac{r^*}{a}\right).$ We now choose 
\bea
h(r^*)=\frac{1}{4}(1-\delta)P\chi_a (r^*). \label{defh}
\eea
Here $\delta>0$ will be chosen later on. We compute 
\bea
h'&=&\frac{1}{4}(1-\delta)P'\chi_a+\frac{1}{4a}(1-\delta)P \chi'_a,\\
h''&=&\frac{1}{4}(1-\delta)P''\chi_a+\frac{1}{2a}(1-\delta)P'\chi_a'+\frac{(1-\delta)}{4a^2}P\chi_a''.
\eea
Here $\chi_a'=(\chi')_a,\, \chi''_a=(\chi'')_a$ and thus there exists
a constant $C>0$ such that
\begin{equation*}
\forall r^*\in \R,\, a>0,\quad \vert \chi'_a(r^*)\vert \le C,\quad
\vert \chi''_a(r^*)\vert \le C.
\end{equation*}
We have 
\begin{equation*}
P'=PV,\, P''=PV^2 + PV'=
PV^2+2P\frac{N}{r^2}(1-3\mu)=PV^2+2P^2(1-3\mu).   
\end{equation*}
\subsection{Estimate of the nonlinear contribution}
\label{NLContr}
We have 
\begin{equation*}
-\frac{1}{4}\int h''(W^2-1)^2=\frac{1-\delta}{16}\int P(\chi_a(-V^2-V')-\frac{2}{a}V\chi_a'-\frac{1}{a^2}\chi_a'')(W^2-1)^2.
\end{equation*}
Apart from an error term which is small when the energy is small and
which will be treated in Section \ref{sec2.6},
the nonlinear contribution is
\begin{eqnarray*}
\lefteqn{\frac{1}{4}\int_{\R} P(-Vf+(\chi_a(1-\delta)-1)P)(W^2-1)^2dr^*}\\
&+&\frac{1-\delta}{16}\int_{\R}
P(\chi_a(-V^2-V')-\frac{2}{a}V\chi_a'-\frac{1}{a^2}\chi_a'')(W^2-1)^2dr^*\\
&=&\int_{\R} \frac{1}{4}P\chi_a(-Vf-\delta
P-\frac{1-\delta}{4}(V'+V^2))(W^2-1)^2dr^*\\
&+&\int_{\R} \frac{1}{4}P\{(1-\chi_a)(-Vf-P)-\frac{1-\delta}{2a}V\chi_a'-\frac{1-\delta}{4a^2}\chi_a''\}(W^2-1)^2dr^*.
\end{eqnarray*}
\begin{lemma}
\label{prop2}
Let $1>\epsilon>0$. Then for $a$ large enough we have uniformly in $0<\delta<1/2$: 
\begin{equation*}
-Vf-\frac{(1-\delta)}{2a}V\chi_a'-\frac{1-\delta}{4a^2}\chi_a''\ge -(1-\epsilon)Vf.
\end{equation*}
\end{lemma}
\proof 
\begin{itemize}
\item We first consider 
\begin{equation*}
-\frac{\epsilon}{2}Vf-\frac{(1-\delta)}{2a}V\chi_a'=V(-\frac{\epsilon}{2}f-\frac{(1-\delta)}{2a}\chi_a').
\end{equation*}
\begin{itemize}
\item For $\vert r^*\vert\le a$
  $V(-\epsilon/2f-\frac{1-\delta}{2a}\chi_a')=-\epsilon/2Vf$ is
  positive. 
\item For $r^*\le -a$ and $a$ large enough $V(r^*)>0$. As
  $-\epsilon/2f(2m)=\frac{\epsilon}{12m}$ the whole expression is
  positive for a large enough. 
\item For $r^*\ge a$ and $a$ large enough $V(r^*)<0.$ Noting that
  $\lim_{r\rightarrow \infty} f(r)=\frac{1}{3m}$ we see that the whole
  expression is positive for $a$ large enough. 
\end{itemize}
\item We now consider
\begin{equation*}
-\frac{\epsilon}{2}Vf-\frac{1-\delta}{4a^2}\chi''_a
\end{equation*}
For $r^*\ge a$ and $a$ large enough we have 
\begin{equation*}
-Vf\gtrsim \frac{1}{a}
\end{equation*}
which gives the estimate for $r^*\ge a$, a large. For $r^*\le -a$ we
use that $-Vf(2m)=\frac{1}{12m^2}$. 
\end{itemize}
\qed 
\begin{lemma}
\label{lem1}
Let $1>\epsilon>0$. For $a$ large enough we have 
\begin{equation*}
(-(1-\epsilon)Vf-P)(1-\chi_a)\ge -(1-2\epsilon)Vf(1-\chi_a).
\end{equation*}
\end{lemma}
\proof

We first consider the case $r^*\ge a$. We have 
\begin{equation*}
-Vf=\frac{2}{3mr}+\cO(r^{-2}),\, P=\cO(r^{-2}),\quad r\rightarrow \infty
\end{equation*}
which shows the inequality for $r^*\ge a$ and $a$ sufficiently
large. For $r^*\le -a$ and $a$ large the inequality follows from 
\begin{equation*}
-Vf(2m)=\frac{1}{12m^2},\, P(2m)=0.
\end{equation*}
\qed

We have the following 
\begin{proposition}
\label{prop3}
There exists $1/2>\epsilon>0$ with the following property. For all $a>0$
there exists $\delta=\delta(a)>0$ such that we have 
\begin{equation*}
\chi_a(-(1-\epsilon)Vf-\delta
P-\frac{1}{4}(1-\delta)V^2-\frac{1}{4}(1-\delta)V')\gtrsim \chi_a.
\end{equation*} 
\end{proposition}
\proof First recall that 
\begin{equation*}
V'=\frac{2N}{r^2}(1-3\mu).
\end{equation*}
We claim that it is sufficient to show
\begin{equation}
\label{pos}
-4Vf-V^2-\frac{2N}{r^2}(1-3\mu)>0\quad\mbox{on}\quad \R_r^{\ge 2m}.
\end{equation}
Let us first show that Proposition \ref{prop3} follows from \eqref{pos}. To see this, we first notice that if \eqref{pos} is satisfied, we also
have 
\begin{equation}
\label{pos1}
-4(1-\epsilon)Vf-V^2-\frac{2N}{r^2}(1-3\mu)>0\quad\mbox{on}\quad \R_r^{\ge 2m}
\end{equation}
for $\epsilon>0$ small enough. Indeed, there exists $C_0>3m$ such that for
$r\ge C_0$, we have 
\begin{equation*}
-4(1-\epsilon)Vf-V^2-\frac{2N}{r^2}(1-3\mu)\ge (-3+4\epsilon)Vf>0
\end{equation*}
if $\epsilon<3/4.$
This follows from
\begin{eqnarray*}
-Vf&=&\frac{2}{3mr}+\cO(r^{-2}),\, r\rightarrow \infty,\\
-V^2&=&\cO(r^{-2}),\, r\rightarrow \infty,\\
-\frac{2N}{r^2}(1-3\mu)&=&\cO(r^{-2}),\, r\rightarrow \infty. 
\end{eqnarray*}
On the compact interval $[2m,C_0]$, we can add a small (negative) multiple of $-fV$
without changing the positivity. Thus, we have proved that \eqref{pos} implies \eqref{pos1} for $\epsilon>0$ small enough, and now, we would like to prove that \eqref{pos1} implies Proposition \ref{prop3}. Let us fix $\epsilon>0$ such that
\eqref{pos1} is satisfied. 
Then, the left hand side of \eqref{pos1}, is $\ge\epsilon_a>0$ on the support of $\chi_a$
for some constant $\epsilon_a$. However, on the support of $\chi_a$ the
functions $V',-P$ are bounded, so we can add small
multiplies of them so as the whole expression would still be $\ge \epsilon_a/2$. Thus, this shows that \eqref{pos} implies Proposition \ref{prop3}. Hence, it
remains to prove \eqref{pos}. In fact, we have 
\begin{eqnarray*}
-4Vf-V^2-\frac{2}{r^2}(1-3\mu)N&=&-4\frac{3\mu-2}{r}(\frac{1}{3m}-\frac{1}{r})-\frac{(6m-2r)^2}{r^4}\\
&-&\frac{2}{r^2}(1-\mu)(1-3\mu).
\end{eqnarray*}
Therefore, it suffices to show that
\begin{eqnarray*}
(8r^2-24mr)(\frac{r}{3m}-1)-(6m-2r)^2-2(r-2m)(r-6m)&>&0\quad\forall r\in \R^{\ge 2m}\\
\Leftrightarrow 8r^3-66mr^2+192m^2r-180m^3&>&0\quad \forall r\in \R^{\ge 2m}.
\end{eqnarray*}
Let $g(r)=8r^3-66mr^2+192m^2r-180m^3.$ We have 
\begin{equation*}
g(2m)=4m^3,\quad \lim_{r\rightarrow\infty}g(r)=\infty.
\end{equation*}
We compute 
\begin{equation*}
g'(r)=24(r^2-\frac{11}{2}mr+8m^2)>0.
\end{equation*}
It follows that $g(r)>0$ for all $r\ge 2m$. 
\qed
\begin{proposition}
\label{prop4}
For $a>0$ large enough and $\delta$ small enough we have 
\begin{eqnarray*}
\lefteqn{\chi_a(-Vf-\delta
P-\frac{1}{4}(1-\delta)V^2-\frac{1}{4}(1-\delta)V')}\\
&+&(1-\chi_a)(-Vf-P)-\frac{1-\delta}{2a}V\chi_a'-\frac{1-\delta}{4a^2}\chi_a''\gtrsim \frac{1}{r}.
\end{eqnarray*}
\end{proposition}

\proof 
We first choose $\epsilon>0$ as in Proposition \ref{prop3}. Then for $a$
large enough we have by Lemma \ref{prop2} uniformly
  in $0<\delta<1/2$:
\begin{equation}
\label{prop4.1}
-Vf -\frac{1-\delta}{2a}V\chi_a'-\frac{1-\delta}{4a^2}\chi_a''\ge -(1-\epsilon)Vf.
\end{equation}
Using Lemma \ref{lem1} we obtain by choosing $a$ possibly larger :
\begin{equation}
\label{prop4.2}
(1-\chi_a)(-(1-\epsilon)Vf-P)\ge -(1-2\epsilon)Vf(1-\chi_a). 
\end{equation}
For $a$ large enough we have for $r^*\ge a$ 
\begin{equation}
\label{prop4.3}
-Vf\gtrsim \frac{1}{r}
\end{equation} 
and for $r^*\le -a$
\begin{equation}
\label{prop4.4}
-Vf\gtrsim 1\gtrsim \frac{1}{r}.
\end{equation}
We fix $a$ such that \eqref{prop4.1}-\eqref{prop4.4} are fulfilled. We
now apply Proposition \ref{prop3} and obtain
by choosing $\delta>0$
small enough :
\begin{eqnarray*}
\lefteqn{\chi_a(-Vf-\delta
P-\frac{1}{4}(1-\delta)V^2-\frac{1}{4}(1-\delta)V')}\\
&+&(1-\chi_a)(-Vf-P)-\frac{1-\delta}{2a}V\chi_a'-\frac{1-\delta}{4a^2}\chi_a''\gtrsim -Vf(1-\chi_a)+\chi_a.
\end{eqnarray*}
Using \eqref{prop4.3} and \eqref{prop4.4} we see that 
\begin{equation*}
-Vf(1-\chi_a)+\chi_a\gtrsim \frac{1}{r}.
\end{equation*}
\qed

\subsection{Estimates on the linear contribution}
We now fix $a,\, \delta$ as in Proposition \ref{prop4}.

\subsubsection{Space derivatives}
\label{SpaceDer}
We have from the definition of $h$ in \eqref{defh},
\begin{eqnarray*}
&& \int_{\R} 
  (W')^2h(W^2-1) dr^* +2\int_{\R}  h(W')^2W^2 dr^* +  \frac{1}{2}\int_{\R}  P(W')^2 dr^* \\
&=&\frac{1}{4}(1-\delta)\int_{\R}  P(W')^2(W^2-1)\chi_a dr^* +\frac{1}{2}(1-\delta)\int_{\R} 
P(W')^2(W^2-1)\chi_a dr^* \\
&& + \frac{1}{2}(1-\delta)\int_{\R} 
P(W')^2\chi_a dr^* + \frac{1}{2}\int_{\R}  P(W')^2 dr^*\\
&=&\frac{3}{4}(1-\delta)\int_{\R} 
P(W')^2(W^2-1) \chi_a dr^* + (1-\frac{1}{2}\delta)\int_{\R}  P(W')^2\chi_a dr^* \\
&& +\frac{1}{2}\int_{\R}  P(W')^2(1-\chi_a) dr^*\\
&\gtrsim&\int_{\R}  P(W')^2 dr^* -\left\vert \int_{\R}  P(W')^2(W^2-1)\chi_a dr^* \right\vert.
\end{eqnarray*}
\subsubsection{Time derivatives}
\label{TimeDer}
We have from the definition of $h$ in \eqref{defh},
\begin{eqnarray*}
&&-\int_{\R}  h\dot{W}^2(W^2-1) dr^* -2\int_{\R}  h\dot{W}^2W^2 dr^* + \frac{1}{2}\int_{\R}  P\dot{W}^2 dr^*  \\
&=&-\frac{3}{4}(1-\delta) \int_{\R} 
P\chi_a\dot{W}^2(W^2-1) dr^* - \frac{1}{2}(1-\delta) \int_{\R}  P\chi_a\dot{W}^2 dr^*+\frac{1}{2}\int_{\R} 
P±\dot{W}^2  dr^* \\
&=&-\frac{3}{4}(1-\delta) \int_{\R} 
P\chi_a\dot{W}^2(W^2-1) dr^* + \frac{1}{2}\delta \int_{\R}  P\chi_a\dot{W}^2 dr^*+\frac{1}{2}\int_{\R} 
P±\dot{W}^2(1-\chi_a) dr^* \\
&\gtrsim&\int_{\R}  P\dot{W}^2 dr^*-\left\vert \int_{\R}  P\chi_a\dot{W}^2(W^2-1) dr^*\right\vert
\end{eqnarray*}
\subsection{Small energy error terms}
\label{sec2.6}
We again fix $a,\, \delta$ as in Proposition \ref{prop4}. 
\begin{proposition}
\label{prop5}
For all $\epsilon_0>0$ there exists $\epsilon_1>0$ such that for $
E_F^{(\frac{\pa}{\pa t})} \le \epsilon_1$ we have 
\begin{eqnarray}
\label{error1}
\left\vert \int_{\R}  P\chi_a(W')^2(W^2-1) dr^* \right\vert&\le& \epsilon_0 \int_{\R}  P
(W')^2 dr^*,\\
\label{error2}
\left\vert \int_{\R}  \chi_aP\dot{W}^2(W^2-1) dr^*\right\vert &\le& \epsilon_0 \int_{\R} 
P\dot{W}^2 dr^*,\\ 
\label{error3}
\left\vert \int_{\R} P^2\chi_a(W^2-1)^3 dr^* \right\vert&\le& \epsilon_0 \int_{\R}  \frac{P}{r}(W^2-1)^2 dr^* \label{treatederrorterm}.
\end{eqnarray}
\proof 

We have 
\begin{eqnarray*}
\left\vert \int_{\R}  P\chi_a(W')^2(W^2-1) dr^* \right\vert&\le& \Vert
\chi_a(W^2-1)\Vert_{L^{\infty}}\int_{\R} P (W')^2  dr^*\\
&\lesssim&\Vert
\sqrt{P}(W^2-1)\Vert_{L^{\infty}}\int_{\R}  P (W')^2 dr^* \\
&\lesssim& \left(\sqrt{ E_F^{(\frac{\pa}{\pa t})} } + E_F^{(\frac{\pa}{\pa t})}\right)\int_{\R}  P (W')^2 dr^*.
\end{eqnarray*}
This shows \eqref{error1}. The proof for \eqref{error2} is strictly
analogous. To prove \eqref{error3} we estimate 
\begin{eqnarray*}
\left\vert \int_{\R} P^2\chi_a(W^2-1)^3dr^*\right\vert&\lesssim& \left(\sqrt{ E_F^{(\frac{\pa}{\pa t})} } + E_F^{(\frac{\pa}{\pa t})}\right)  \int_{\R}  \frac{P}{r}(W^2-1)^2 dr^*.
\end{eqnarray*}
\qed
\end{proposition}
\subsection{End of the proof of the estimate \eqref{Morawetz-estimate}}

We choose $a$ large enough and $\delta>0$ small enough such that
the estimate in Proposition \ref{prop4} is fulfilled. Once $a,\delta$
fixed in this way we choose the energy small enough such that
\begin{enumerate}
\item The nonlinear contribution in Section \ref{NLContr} plus the error term
  $\int Ph(W^2-1)^3$ dominate 
\begin{equation*}
\int_{\R} \frac{P}{r} (W^2-1)^2 dr^*.
\end{equation*}
\item The space derivatives in Section \ref{SpaceDer} dominate 
\begin{equation*}
\int P (W')^2 dr^*. 
\end{equation*}
\item The time derivatives in Section \ref{TimeDer} dominate 
\begin{equation*}
\int P\dot{W}^2 dr^*.
\end{equation*}
\end{enumerate} 

Then,
integrating in $t$, we obtain that
for initial data with small enough energy,
\begin{eqnarray*}
&& \int _{t=t_{0}}^{t} \int_{\R}
P\dot{W}^2+P(W')^2+\frac{P}{r}(W^2-1)^2 dr^* \\
&\lesssim& \left\vert \int_{\R}
\dot{W}P\chi_aW(W^2-1) dr^* \right\vert^{t}_{t=t_{0}}+ \left\vert \int_{\R} \dot{W}fW' dr^* \right\vert^{t}_{t=t_{0}}.
\end{eqnarray*}
We have by the Cauchy-Schwarz inequality
\begin{eqnarray*}
\left\vert \int_{\R} \dot{W}fW' dr^* \right\vert &\lesssim& \left(\int_{\R}  \dot{W}^2 dr^* \right)^{1/2}\left(\int_{\R} 
  (W')^2 dr^* \right)^{1/2}\lesssim E_F^{(\frac{\pa}{\pa t})} (t=t_{0}),\\
\left\vert \int_{\R}  \dot{W}P\chi_aW(W^2-1) dr^* \right\vert&\lesssim&\left(\int_{\R} 
  \dot{W}^2 dr^* \right)^{1/2}\left(\int_{\R} 
  P^2(W^2-1)^2\chi_a^2W^2 dr^* \right)^{1/2}\\
&\lesssim& \left(\int_{\R} 
  \dot{W}^2 dr^* \right)^{1/2}\Vert P\chi_a^2W^2\Vert_{L^{\infty}}^{1/2}\left(\int_{\R} 
  P(W^2-1)^2 dr^* \right)^{1/2} \\
&\lesssim& E_F^{(\frac{\pa}{\pa t})} (t=t_{0})(E_F^{(\frac{\pa}{\pa t})} (t=t_{0})+\sqrt{E_F^{(\frac{\pa}{\pa t})} (t=t_{0})}+1).
\end{eqnarray*}
where in the last inequality we have used Proposition \ref{prop1} for energy small enough. This finishes the proof of the estimate. 
\qed

\section{The Proof of Local Energy Decay on $t= constant$
  Hypersurfaces}
\label{sec3}
The proof of the local energy decay follows the arguments of
\cite{G2} which we adapt to the situation where the Morawetz estimate
is only available et low energies. We refer to \cite{G2} for the
details of the calculations. 

\subsection{The vector field $K$}\

Let
\bea
K &=& - w^{2} \frac{\pa}{\pa w} - v^{2} \frac{\pa}{\pa v}
\eea
We then compute (see \cite [equation (44)]{G2})
\bea
\notag
\pi^{\a\b}(K)T_{\a\b}(F) &=&   4t  [ 2 + \frac{ ( 3\mu - 2 )r^{*}}{r} ] .[  | F_{\hat{v}\hat{w}}|^{2}  +  \frac{1}{4 } | F_{\hat{\phi}\hat{\th}}|^{2} ]  \\ \label{zerocomponontsconformalenergy}
\eea
and define
\bea
\notag
J_{F}^{(K)} ( t_{i} \leq t \leq t_{i+1} )  = \int_{t = t_{i} }^{ t = t_{i+1} }  \int_{r^{*} = - \infty}^{r^{*}= \infty} \int_{\S^{2}} \pi^{\a\b}(K)T_{\a\b}(F) dVol \\
\eea

We also have \cite[equation (47)]{G2} :
\bea
\notag
E_{F}^{(K)}(t_{i}) &=& \int_{r^{*} = - \infty}^{r^{*}= \infty} \int_{\S^{2}}   [   w^{2}N ( | F_{\hat{w}\hat{\th}} |^{2} +  | F_{\hat{w}\hat{\phi}} |^{2} ) +  v^{2} ( \frac{1}{N} | F_{\hat{v}\hat{\th}} |^{2} + \frac{1}{N} | F_{\hat{v}\hat{\phi}} |^{2} )  \\
&&+  (w^{2} + v^{2} ) (  | F_{\hat{v}\hat{w}}|^{2}   +  \frac{1}{4} | F_{\hat{\phi}\hat{\th}}|^{2} ) ]   r^{2} N d\sigma^{2} dr^{*} 
\eea

\subsection{Local energy decay}
The following estimate controls the bulk term :
\begin{proposition}
\bea
\notag
J_{F}^{(K)}( t_{i} \leq t \leq t_{i+1})    &\lesssim&  t_{i+1}     \int_{t_i}^{t_{i+1}}\int_{r^*=r_{0}^*}^{R_{0}^*} \int_{\S^2}  \frac{N}{r} ( | F_{\hat{w}\hat{v}} |^2 +| F_{\hat{\th}\hat{\phi}}  |^2 ) r^2  d\si^2 dr^* dt \\ \label{KcontolledbytG}
\eea
where $ 2m < r_{0} \leq 3m \leq R_{0}$
\end{proposition}

\begin{proof}\
First note that we have for $\vert r_*\vert >>1$:
\begin{equation*}
2+\frac{(3\mu-2)r_*}{r}<0.
\end{equation*}
On the remaining compact interval $\frac{N}{r}$ is strictly positive. 
\end{proof}

We now estimate the local energy in terms of $E_{F}^{(K)}$

\begin{proposition}
\bea
E_{F}^{(\frac{\pa}{\pa t})}  ( r^*_1 \leq r^* \leq r^*_2 )  &\leq& C(r_{1}^*, r_{2}^*)  \frac{E_{F}^{(K)}(t)}{ t^{2}} \label{local energy in terms of conformal energy}
\eea
\end{proposition}

\begin{proof}
We have for $t$ sufficiently large :
\bea
\notag
&& E_{F}^{(\frac{\pa}{\pa t})}  ( r^*_1 \leq r^* \leq r^*_2 ) (t) \\
\notag
&=&\int_{r^{*}= r_{1}^{*}   }^{r^{*} = r_{2}^{*}   } \int_{\S^{2}} ( N | F_{\hat{w}\hat{\th}} |^{2} + N | F_{\hat{w}\hat{\phi}} |^{2}   +   \frac{1}{N} | F_{\hat{v}\hat{\th}} |^{2} + \frac{1}{N} | F_{\hat{v}\hat{\phi}} |^{2}  +   |F_{\hat{v}\hat{w}}|^{2}   +  \frac{1}{4} | F_{\hat{\phi}\hat{\th}}|^{2} ). N r^{2}   d\sigma^{2} dr^{*} (t) \\
& \lesssim & \frac{E_{F}^{(K)}(t)}{\min_{\{  r_{1}^{*} \leq r^{*} \leq r_{2}^{*} \}} w^{2}(t,r_*)}  + \frac{E_{F}^{(K)}(t)}{\min_{\{  r_{1}^{*} \leq r^{*} \leq r_{2}^{*} \}  } v^{2}(t,r_*)} \label{EKoverminvsquaredandwsquared} 
\eea

\end{proof}

We eventually obtain the following local energy decay 

\begin{proposition}
\label{prop3.5}
We have
\bea
E_{F}^{(\frac{\pa}{\pa t})}  ( r^*_1 \leq r^* \leq r^*_2 ) &\leq& C(r_{1}^*, r_{2}^*)  \frac{(E_{F}^{(K)}(t_{0}) + |E_{F}^{(\frac{\pa}{\pa t})} (t_{0})| ) }{ t} \label{local energy decay}
\eea

\end{proposition}

\begin{proof}
Using \eqref{local energy in terms of conformal energy}, \eqref{KcontolledbytG}, the divergence theorem, and the Morawetz estimate \eqref{th2}, we get \eqref{local energy decay}.
\end{proof}
\subsection{Decay for the middle components away from the horizon} \

\begin{proposition}
Let $R > 2m$. We have for all $r \geq R $,

\beaa
  |F_{\hat{\th} \hat{\phi}} (v, w, \th, \phi) | =  \left\vert\frac{W^2 (v, w) -1}{r^2}\right\vert  &\lesssim & \frac{ E_{1}  }{ \sqrt{ 1 + |  v | }  }   
 \eeaa
 and
\beaa
  |F_{\hat{\th} \hat{\phi}} (v, w, \th, \phi) | =  \left\vert\frac{W^2 (v, w) -1}{r^2}\right\vert  &\lesssim & \frac{ E_{1}  }{ \sqrt{ 1 +  | w |  }  }   
 \eeaa
where,
\beaa
E_{1} &= & [    | E_{    F }^{(\frac{\pa}{\pa t}) } (t=t_{0}) |  +    E_{ F }^{(K) } (t_{0})    ]^{\frac{1}{2}}  
\eeaa
\end{proposition}

\begin{proof}\

We consider the region $w \geq 1$, $r \geq R$, where $R$ is fixed.\\

Let, $r_{F}$ be a value of $r$ such that $ R^* \leq r_{F}^* \leq R^*+1 $, and to be determined later. Let $r^{*}  \geq r_{F}^*$. We have, 
\beaa
\int_{\S^{2}} r^{2} |F_{\hat{\th} \hat{\phi}}|^{2}(t, r, \th, \phi)  d\sigma^{2} &=&  \int_{\S^{2}} r^{2} |F_{\hat{\th} \hat{\phi}} |^{2}(t, r_{F}, \th, \phi)  d\sigma^{2}   +   \int_{\S^{2}} \int_{\overline{r}^{*} = r^{*}_{F}}^{\overline{r}^{*} = r^{*} } \pa_{r^{*}} [r^{2} | F_{\hat{\th} \hat{\phi}} |^{2}] (t, r, \th, \phi)  d\overline{r}^{*} d\sigma^{2} \\
&=&  \int_{\S^{2}} r_{F}^{2} |F_{\hat{\th} \hat{\phi}}|^{2}(t, r_{F}, \th, \phi)  d\sigma^{2}   +   \int_{\S^{2}} \int_{\overline{r}^{*} = r^{*}_{F}}^{\overline{r}^{*} = r^{*} } 2r | F_{\hat{\th} \hat{\phi}} |^{2} (t, r, \th, \phi) N  d\overline{r}^{*} d\sigma^{2}  \\
&&+  2   \int_{\S^{2}} \int_{\overline{r}^{*} = r^{*}_{F}}^{\overline{r}^{*} = r^{*} } r^{2}  \pa_{r^{*}} |F_{\hat{\th} \hat{\phi}}  |^{2} (t, r, \th, \phi)  d\overline{r}^{*} d\sigma^{2}  \\
\eeaa

From \eqref{EKoverminvsquaredandwsquared} we obtain
\begin{equation}
\label{localenergyboudedbyconformalonewithexplicitconstants}
\int_{r^{*}= r_{1}^{*}   }^{r^{*} = r_{2}^{*}   }
  \int_{\S^{2}}  |F_{\hat{\th} \hat{\phi}} |^{2} (t, \overline{r},
  \th, \phi) N r^{2} d\sigma^{2}  d\overline{r}^{*}  \lesssim
  \frac{E_{F}^{(K)}(t)}{\min_{\{  r_{1}^{*} \leq r^{*} \leq r_{2}^{*}
      \}  } w^{2}}  + \frac{E_{F}^{(K)}(t)}{\min_{\{  r_{1}^{*} \leq
      r^{*} \leq r_{2}^{*} \}  } v^{2}}
\end{equation}

Therefore,

$$ \int_{\overline{r}^{*} = R^{*} }^{\overline{r}^{*} = R^*+1 }    \int_{\S^{2}}  |F_{\hat{\th} \hat{\phi}} |^{2} (t, \overline{r},\th, \phi) N r^{2} d\sigma^{2}  d\overline{r}^{*}  \lesssim \frac{  E_{F}^{(K)} (t)}{t^{2}} $$

There exists $r_{F}$, such that $R^* \leq r_{F}^* \leq R^*+1$ and,
\beaa
 \int_{\S^{2}}  r_{F}^{2} |F_{\hat{\th} \hat{\phi}} |^{2} (t, r_{F}, \th, \phi) d\sigma^{2}   &\lesssim& \frac{  E_{F}^{(K)} (t)}{t^{2}(R^*+1-R^*) N(R)} 
 \eeaa
which gives,
\bea
\int_{\S^{2}}  r_{F}^{2} |F_{\hat{\th} \hat{\phi}} |^{2} (t, r_{F}, \th, \phi) d\sigma^{2} &\lesssim& \frac{ E_{F}^{(K)} (t)}{t^{2}} \label{estimateonthefirstterminthesobolevineq}
\eea

On the other hand, from \eqref{localenergyboudedbyconformalonewithexplicitconstants} and from looking at the region of integration in Penrose diagram, it is easy to see that

\bea
\notag
 \int_{\overline{r}^{*} = r_{F}^{*} }^{\overline{r}^{*} = r^{*} }  \int_{\S^{2}}  \overline{r}  |F_{\hat{\th} \hat{\phi}} |^{2} N (t, \overline{r}, \th, \phi) d\sigma^{2}  d\overline{r}^{*}  &\lesssim&
\int_{\overline{r}^{*} = r_{F}^{*} }^{\overline{r}^{*} = r^{*} }  \frac{\overline{r}}{R}   \int_{\S^{2}}  \overline{r} |F_{\hat{\th} \hat{\phi}} |^{2} N (t, \overline{r}, \th, \phi) d\sigma^{2}  d\overline{r}^{*} \\
& \lesssim& \frac{  E_{F}^{(K)} (t)}{t^{2}}  + \frac{  E_{F}^{(K)} (t)}{w^{2}} 
\eea

Thus,
\bea
\int_{\overline{r}^{*} = r_{F}^{*} }^{\overline{r}^{*} = r^{*} }  \int_{\S^{2}}  \overline{r} |F_{\hat{\th} \hat{\phi}} |^{2} N (t, \overline{r}, \th, \phi) d\sigma^{2}  d\overline{r}^{*}  &\lesssim& \frac{  E_{F}^{(K)} (t)}{t^{2}}  + \frac{  E_{F}^{(K)} (t)}{w^{2}}  \label{estimateonthesecondterminthesobolevineq}
\eea

Now, we want to estimate the term:
\beaa
&& \int_{\S^{2}} \int_{\overline{r}^{*} = r^{*}_{F}}^{\overline{r}^{*} = r^{*} } r^{2}  \der_{r^{*}} |F_{\hat{\th} \hat{\phi}}|^{2} (t, r, \th, \phi)  d\overline{r}^{*} d\sigma^{2} 
\eeaa

Using Cauchy-Schwarz, we obtain

\beaa
&& \int_{\S^{2}} \int_{\overline{r}^{*} = r^{*}_{F}}^{\overline{r}^{*} = r^{*} } r^{2}  \pa_{r^{*}} |F_{\hat{\th} \hat{\phi}}|^{2} (t, r, \th, \phi) d\overline{r}^{*} d\sigma^{2} \\
&\lesssim&  \left(\int_{\overline{r}^{*} = r^{*}_{F}}^{\overline{r}^{*} = r^{*} }  \int_{\S^{2}}  \overline{r}^{2}  | \pa_{r} F_{\hat{\th} \hat{\phi}} |^{2} N^2 d\sigma^{2}  d\overline{r}^{*} \right)^{\frac{1}{2}}  \left(\int_{\overline{r}^{*} = r^{*}_{F}}^{\overline{r}^{*} = r^{*} }   \int_{\S^{2}}  \overline{r}^{2}  | F_{\hat{\th} \hat{\phi}} |^{2}  d\sigma^{2}  d\overline{r}^{*} \right)^{\frac{1}{2}}
\eeaa

We have,
\beaa
 \int_{\overline{r}^{*} = r_{F}^{*} }^{\overline{r}^{*} = r^{*} }  \int_{\S^{2}}  \overline{r}^{2} |F_{\hat{\th} \hat{\phi}}|^{2}  (t, \overline{r},\th, \phi) d\sigma^{2}  d\overline{r}^{*}  & \lesssim& \frac{  E_{F}^{(K)} (t)}{t^{2}}  +  \frac{  E_{F}^{(K)} (t)}{w^{2}} 
\eeaa

Thus,
\bea
[\int_{\overline{r}^{*} = r^{*}_{F}}^{\overline{r}^{*} = r^{*} }   \int_{\S^{2}}  \overline{r}^{2}  |F_{\hat{\th} \hat{\phi}} |^{2}  d\sigma^{2}  d\overline{r}^{*} ]^{\frac{1}{2}} \lesssim \frac{  \sqrt{E_{F}^{(K)} (t)}}{t}  + \frac{  \sqrt{E_{F}^{(K)} (t)}}{w} 
\eea

On the other hand, we have

\beaa
&& \int_{\S^{2}} \int_{\overline{r}^{*} = r^{*}_{F}}^{\overline{r}^{*} = r^{*} } r^{2}   |\pa_{r} F_{\hat{\th} \hat{\phi}}|^{2} (t, r, \th, \phi)  N^2 d\overline{r}^{*} d\sigma^{2} \\
&=& \int_{\S^{2}} \int_{\overline{r}^{*} = r^{*}_{F}}^{\overline{r}^{*} = r^{*} }  r^{2} [  \frac{-2}{r^3}(W^2-1) + 2\frac{W \pa_{r} W}{r^2}  ]^2 (t, r, \th, \phi)N^2 d\overline{r}^{*} d\sigma^{2} \\
 &\les& \int_{\S^{2}} \int_{\overline{r}^{*} = r^{*}_{F}}^{\overline{r}^{*} = r^{*} }  [    \frac{N^2}{r^4} (W^2-1)^2 + \frac{ N^2(W^2 - 1 +1)}{r^2} |\pa_{r} W |^{2}  ] d\overline{r}^{*} d\sigma^{2} \\
 &\les& \int_{\S^{2}} \int_{\overline{r}^{*} = r^{*}_{F}}^{\overline{r}^{*} = r^{*} }  [    \frac{N}{r^2} (W^2-1)^2 + \frac{(W^2 - 1 +1) }{r^2}|\pa_{r} W |^{2} N^2 ] d\overline{r}^{*} d\sigma^{2} \\
  &\les&   (\Vert
  \sqrt{P}(W^2-1)\Vert_{L^{\infty}(\R)}+1)\int_{\S^{2}}
  \int_{\overline{r}^{*} = r^{*}_{F}}^{\overline{r}^{*} = r^{*} }  [
  \frac{N}{r^2} (W^2-1)^2 + |\pa_{r^*} W |^{2} ] d\overline{r}^{*}
  d\sigma^{2}\\
&\les&(E^{(\frac{\pa}{\pa_t})}_F+\sqrt{E^{(\frac{\pa}{\pa_t})}_F}+1) \int_{\overline{r}^{*} = r^{*}_{F}}^{\overline{r}^{*} = r^{*} }  [
  \frac{N}{r^2} (W^2-1)^2 + |\pa_{r^*} W |^{2} ] d\overline{r}^{*}
  d\sigma^{2}.
\eeaa
Here we have used Proposition \ref{prop1}. Thus, by using \eqref{localenergyboudedbyconformalonewithexplicitconstants} and from looking again at the region of integration in Penrose diagram, we obtain

\bea
 \int_{\S^{2}} \int_{\overline{r}^{*} = r^{*}_{F}}^{\overline{r}^{*} = r^{*} } r^{2}   |\pa_{r} F_{\hat{\th} \hat{\phi}}|^{2} (t, r, \th, \phi) N^2 d\overline{r}^{*} d\sigma^{2} & \lesssim& \frac{  E_{F}^{(K)} (t)}{t^{2}}  + \frac{  E_{F}^{(K)} (t)}{w^{2}} 
\eea

Finally, we obtain,
\beaa
 r^{2} |F_{\hat{\th} \hat{\phi}} |^{2} (t, r, \th, \phi)  & \lesssim& \frac{  E_{F}^{(K)} (t)}{t^{2}}  + \frac{  E_{F}^{(K)} (t)}{w^{2}} 
\eeaa
 Thus,
\beaa
 |F_{\hat{\th} \hat{\phi}}|(v, w, \th, \phi)  &\lesssim& \frac{ [ E_{  F}^{(K)} (t) ]^{\frac{1}{2}} }{ r t }  + \frac{  [ E_{  F}^{(K)} (t) ]^{\frac{1}{2}}  }{ r w } 
\eeaa

 Thus,
\beaa
 |F_{\hat{\th} \hat{\phi}}|(v, w, \th, \phi)  &\lesssim& \frac{ [ E_{  F}^{(K)} (t) ]^{\frac{1}{2}} }{ r t }  + \frac{  [ E_{  F}^{(K)} (t) ]^{\frac{1}{2}}  }{ r w } 
\eeaa
Using \eqref{KcontolledbytG} and the Morawetz estimate we see that 
\begin{equation*}
[ E_{  F}^{(K)} (t) ]^{\frac{1}{2}} \lesssim E_1\sqrt{t}.
\end{equation*}
If $r^*\le (1-\epsilon)t,\, 1>\epsilon>0$ we have $w=t-r^*\ge \epsilon t$. If $r^*\ge
(1-\epsilon)t$ we have 
\begin{equation*}
\frac{\sqrt{t}}{rw}\le \frac{1}{\sqrt{r}w}. 
\end{equation*}
Summarizing we obtain in the region $w\ge 1,\, r\ge R$:
\begin{equation*}
 |F_{\hat{\th} \hat{\phi}}|(v, w, \th, \phi)  \lesssim \frac{E_1}{\sqrt{ r t} }  + \frac{E_1}{\sqrt{rw}}. 
\end{equation*}
Consider first the part of the region where $t \geq 1$:
\begin{itemize}
\item In the region $r\ge R,\, t\ge 1,\, w\ge 1$ we have 
\begin{equation*}
w=t-r^*\lesssim r+t\lesssim rt
\end{equation*}
and thus we obtain the estimate
\begin{equation*}
|F_{\hat{\th} \hat{\phi}}|(v, w, \th, \phi)|  \lesssim \frac{E_1}{\sqrt{w}}.
\end{equation*}
\item To obtain the estimate in $v$ we distinguish two cases 
\begin{enumerate}
\item $v\le 1$. The region $w\ge 1,\, v\le 1$ is a compact region and
  $|F_{\hat{\th} \hat{\phi}}|(v, w, \th, \phi)|$ is uniformly bounded
  there. 
\item Now consider $v\ge 1,\, w\ge 1,\, r\ge R,\, t\ge 1$. We have 
\begin{eqnarray*}
v=r^*+t\lesssim r+t\lesssim rt,\\
v\lesssim r+t\lesssim Cr+t-r^*\lesssim r+w\lesssim rw.
\end{eqnarray*} 
Thus for $v\ge 1$ we have 
\begin{equation*}
|F_{\hat{\th} \hat{\phi}}|(v, w, \th, \phi)|\lesssim \frac{E_1}{\sqrt{v}}
\end{equation*}
\end{enumerate}
\end{itemize}
For $t \leq 1$, the intersection with the region $w\ge 1,\, r\ge R$, is compact, hence, $|F_{\hat{\th} \hat{\phi}}|(v, w, \th, \phi)|$ is uniformly bounded in this region.
 
Somewhat similar arguments for the other regions, see \cite{G2} for details, gives the stated result.

\end{proof}\

\section{Decay of the Energy to Observers Traveling to the Black Hole on $v = constant$ Hypersurfaces Near the Horizon}\
\label{Sectionhorizon}

\subsection{The vector field $H$}\

Let
\bea
H &=& - \frac{ h(r^{*})}{N} \frac{\pa}{\pa w}  - h(r^{*}) \frac{\pa}{\pa v}  
\eea
We have, see \cite[equation (88)]{G2},
\bea
\notag
&& \pi^{\a\b}(H)T_{\a\b}(F) \\
\notag
&=& ( | F_{\hat{w}\hat{\th}} |^{2} + | F_{\hat{w}\hat{\phi}} |^{2} ) (  h^{'} - \frac{\mu}{r} h )  + (  | F_{\hat{v}\hat{\th}} |^{2} +  | F_{\hat{v}\hat{\phi}} |^{2} ) ( \frac{-1}{N}  h^{'}  )\\
\notag
&& + [  |F_{\hat{v}\hat{w}}|^{2}   +  \frac{1}{4} | F_{\hat{\phi}\hat{\th}}|^{2}  ]   .  [ -2 (      \frac{1}{2N} [ (h^{'} - N h^{'})  - \frac{\mu}{r} h  ])    +      \frac{(2-3\mu)}{Nr}   (h - N h)  ] 
\eea
This gives the following bulk term
\bea
\notag
&& I^{(H)}_{ F} (  v_{i} \leq v \leq v_{i+1} ) ( w_{i} \leq w \leq
\infty)\equiv \int_{v=v_i}^{v=v_{i+1}}\int_{w=w_i}^{w=w_{i+1}}\int_{S^2}\pi^{\alpha\beta}T_{\alpha\beta}(F) dVol \nonumber\\
\notag
&=& \int_{v = v_{i}}^{v = v_{i+1}} \int_{w = w_{i}}^{w = \infty } \int_{\S^{2}} ( [ | F_{\hat{w}\hat{\th}} |^{2} + | F_{\hat{w}\hat{\phi}} |^{2}  ] (    h^{'} - \frac{\mu}{r} h  )   + ( [ | F_{\hat{v}\hat{\th}} |^{2} + | F_{\hat{v}\hat{\phi}} |^{2}  ](  \frac{- h^{'}}{N}  )\nonumber \\
&+& [  |F_{\hat{v}\hat{w}}|^{2}   +  \frac{1}{4} | F_{\hat{\phi}\hat{\th}}|^{2}   ]  . \mu [   \frac{-1}{N}  h^{'}  + \frac{3}{r}   h   ]     ). r^{2} d\sigma^{2} Ndw dv
\eea
We also compute the flux of $H$ along the null hypersurfaces \cite[equation (92)]{G2}
\begin{eqnarray}
\lefteqn{F_{F}^{(H)} ( v = v_{i} ) ( w_{i} \leq w \leq w_{i+1} )}\nonumber\\ 
\notag
&=& \int_{w = w_{i}}^{w = w_{i+1}} \int_{\S^{2}}  - 2 N h (r^{*}) [    | F_{\hat{w}\hat{\th}} |^{2} + | F_{\hat{w}\hat{\phi}} |^{2} +    |F_{\hat{v}\hat{w}}|^{2}   +  \frac{1}{4} | F_{\hat{\phi}\hat{\th}}|^{2}   ] r^{2} d\sigma^{2} d w \\
\end{eqnarray}
and \cite[equation (90)]{G2}
\begin{eqnarray}
\lefteqn{F_{F}^{(H)}  ( w = w_{i} ) ( v_{i} \leq v \leq v_{i+1} )}\nonumber\\
\notag
&=& \int_{v = v_{i}}^{v = v_{i+1}} \int_{\S^{2}}  - 2 h(r^{*}) [   |F_{\hat{v}\hat{w}}|^{2}   +  \frac{1}{4} | F_{\hat{\phi}\hat{\th}}|^{2}    +   | F_{\hat{v}\hat{\th}} |^{2} +  | F_{\hat{v}\hat{\phi}} |^{2}    ] r^{2} d\sigma^{2} d v  \\ \label{firstexpressionforthefluxofHlonwequalconstant}
\end{eqnarray}
We will also need the energy \cite[page 76]{G2}
\beaa
E^{(H)}_{F} (t)  &=& \int_{r^{*}  = -\infty }^{r^{*} = \infty} \int_{\S^{2}}   [ h N(  | F_{\hat{w}\hat{\th}} |^{2} + | F_{\hat{w}\hat{\phi}} |^{2} )  +  h(N+1) (  |F_{\hat{v} \hat{w}}|^{2}   +  \frac{1}{4} | F_{\hat{\phi}\hat{\th}}|^{2} ) \\
&&  +  h (  | F_{\hat{v}\hat{\th}} |^{2} +  | F_{\hat{v}\hat{\phi}} |^{2})  ]   r^{2}  d\sigma^{2}  dr^{*}.   
\eeaa
We will also need the energy 
\begin{eqnarray*}
\lefteqn{E_{F}^{ \# (\frac{\pa}{\pa t} )} (t)}\\
&=& \int_{r^{*} = - \infty }^{r^{*} = \infty } \int_{\S^{2}}    [ N | F_{\hat{w}\hat{\th}} |^{2} + N | F_{\hat{w}\hat{\phi}} |^{2}  +  | F_{\hat{v}\hat{\th}} |^{2} + | F_{\hat{v}\hat{\phi}} |^{2}   +  |F_{\hat{v}\hat{w}}|^{2}   +  \frac{ 1 }{4 } | F_{\hat{\phi}\hat{\th}}|^{2}   ] . r^{2}   d\sigma^{2} dr^{*}\label{defmodifiedenergy} \\
\end{eqnarray*}
Note that $E_F^{\#(\frac{\pa}{\pa t} )}(t)$ is equivalent to
$E^{(H)}_{F} (t)$ in regions where $h=1$. 
We are going to choose $h$ such that

$$ \lim_{r_*\rightarrow -\infty}h(r^{*}) =  1 $$

and for all $r > 2m$ : 
$$ h \geq 0$$
For $2m<r_1$ and $r_1$ sufficiently close to $2m$ we can choose $h$
such that $h$ is supported in $2m< r\le 1.2 r_1$ and such that it
satisfies on $(2m,r_1]$  
\bea
\label{60}
h  &>& 0  \\
\label{61}
h'  &\geq& 0  \\
\label{62}
\frac{\mu}{r} h - h^{'} &\gtrsim& h \\
\label{63}
\mu [\frac{h'}{N}  - \frac{3}{r}   h]   &\gtrsim&  h
\eea
A suitable choice is for example $h(r^*)=exp\left(\int_{-\infty}^{r_*}\delta
\frac{\mu(s)}{r(s)}ds\right)$ ($0<\delta<1$) on $2m<r\le r_1$ and then extended
smoothly.

Applying the divergence theorem for $F_{\mu\nu} $ in a rectangle in
the Penrose diagram representing the exterior of the Schwarzschild
space-time of which one side contains the horizon, say in the region $
[w_{i}, \infty ]\times[ v_{i}, v_{i+1} ] $, then, we have
\begin{eqnarray}
\label{divbase}
\lefteqn{- F_{F}^{(H)} ( w = w_{i} ) ( v_{i} \leq v \leq v_{i+1} ) -   F_{F}^{(H)} ( v = v_{i} ) ( w_{i} \leq w \leq \infty )}\nonumber\\
&=&- I_{F}^{(H)} (  v_{i} \leq v \leq v_{i+1} ) ( w_{i} \leq w \leq \infty)\nonumber\\
&-& F_{F}^{(H)} ( w = \infty ) ( v_{i} \leq v \leq v_{i+1} )   -
F_{F}^{(H)} ( v = v_{i+1} ) ( w_{i} \leq w \leq \infty ). 
\end{eqnarray}
The following figure shows the domain of integration:

\vspace{-1.5cm}
\begin{center}
\includegraphics[width=10cm]{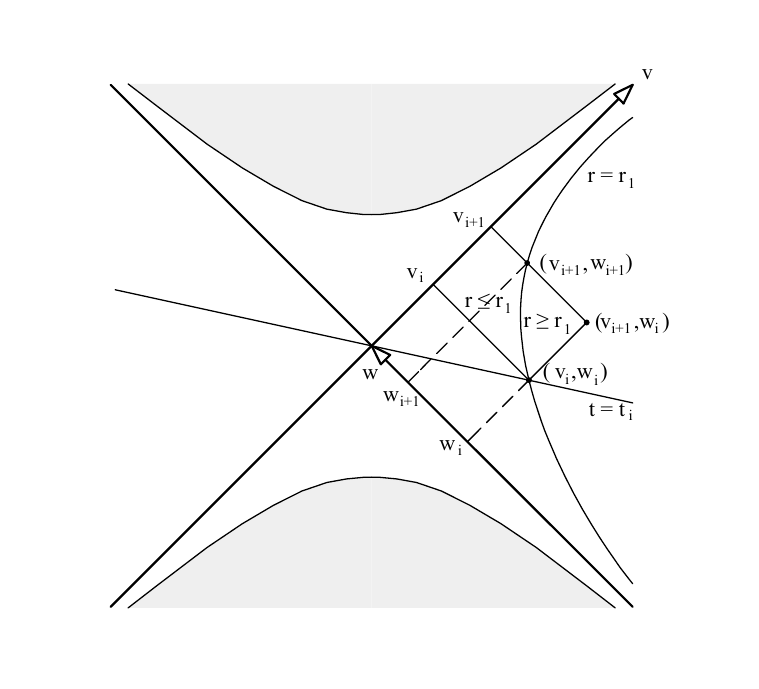}
\end{center}
\vspace{-1.5cm}

\begin{remark}
At this stage, it is not clear that $F_{F}^{(H)} ( w = \infty )(v_i\le
v\le v_{i+1})$ is finite at the regularity level of Corollary \ref{Cor1}. To show that it is well defined,
we can either first work on the regularity level of \cite{G1}, or first work in a region
$[w_{i},w]\times[ v_{i}, v_{i+1} ]$ and then consider the
limit $w\rightarrow \infty$. Our estimates show then in particular that $F_{F}^{(H)} ( w = \infty ) ( v_{i} \leq v \leq v_{i+1} ) $ defined as the above limit is
finite at the regularity level of Corollary \ref{Cor1}. We omit these details.
\end{remark}
\subsection{The Main Estimates}

Let
\beaa
w_{i} &=& t_{i} - r_{1}^{*} \\
v_{i} &=& t_{i} + r_{1}^{*} 
\eeaa
where $t_i$ is a sequence of positive numbers with 
\begin{equation*}
t_i<t_{i+1}\le 1.1
t_i,\quad \sum_{i=0}^{\infty}
\frac{1}{t_i}<\infty
\end{equation*} 
and $r_{1}$ is as determined in the construction of the vector field $H$. We have $r(w_{i}, v_{i}) = r_{1}$.

\subsubsection{Controlling the flux of $H$ away from the horizon}
Away from the horizon we can control the flux of the vector field $H$
by the flux of the vector field $\partial_t$: 
\begin{proposition}
For $v_{i+1} \geq v_{i} $, we have
\end{proposition}

\bea
- F_{F}^{(H)} ( w = w_{i} ) ( v_{i} \leq v \leq v_{i+1} )  \lesssim  F_{F}^{(\frac{\pa}{\pa t})}  ( w = w_{i} ) ( v_{i} \leq v \leq v_{i+1} )  \label{estimate1H}
\eea

\begin{proof}

We have (see \cite[page 67]{G2}),
\beaa
 F_{F}^{(\frac{\pa}{\pa t})}  ( w = w_{i} ) ( v_{i} \leq v \leq v_{i+1} ) &=& \int_{v = v_{i}}^{v = v_{i+1}} \int_{\S^{2}}   2 [   |F_{\hat{v}\hat{w}}|^{2}   +  \frac{1}{4} | F_{\hat{\phi}\hat{\th}}|^{2}    +     \frac{1}{ N} | F_{\hat{v}\hat{\th}} |^{2} + \frac{1}{ N} | F_{\hat{v}\hat{\phi}} |^{2}    ] r^{2} N d\sigma^{2} d v 
\eeaa
The region $w = w_{i}$ and $v_{i} \leq v \leq v_{i+1} $ is in the region $r \geq r_{1}$ as $ r(w_{i}, v_{i}) = r_{1}$, and $v_{i+1} \geq v_{i} $. Thus, in this region 
$$\frac{h(r^{*})}{N}  \lesssim 1 $$

which gives immediately \eqref{estimate1H}.

\end{proof}

\subsubsection{Controlling the bulk term generated from $H$ away from the horizon}
By construction of the vector field $H$ the bulk has good sign in the
region $r\le r_1$. In the region $r\ge r_1$ it can be controlled by
the Morawetz estimate. We obtain :

\begin{proposition} \label{estimateforcompactlysupportedintegratedenergyestimate}
We have
\begin{eqnarray*}
\lefteqn{| I_{F}^{(H)} (  v_{i} \leq v \leq v_{i+1} ) ( w_{i} \leq w \leq \infty) (r \geq r_{1}) |}\\
&\les&  | E^{(\frac{\pa}{\pa t})}_{F}  (  -(0.85)t_{i} \leq r^{*} \leq (0.85)t_{i}  )  (t= t_{i}) |   + \frac{ ( 1+E_{F}^{ \# (\frac{\pa}{\pa t} )} ( t= t_{0} )  + E_{F}^{ (\frac{\pa}{\pa t} )} ( t= t_{0} ) )^2}{t_i}, 
\end{eqnarray*}
where
\begin{eqnarray*}
\lefteqn{E_{F}^{ \# (\frac{\pa}{\pa t} )} ( t= t_{0} )}\\
&=& \int_{r^{*} = - \infty }^{r^{*} = \infty } \int_{\S^{2}}    [ N | F_{\hat{w}\hat{\th}} |^{2} + N | F_{\hat{w}\hat{\phi}} |^{2}  +  | F_{\hat{v}\hat{\th}} |^{2} + | F_{\hat{v}\hat{\phi}} |^{2}   +  |F_{\hat{v}\hat{w}}|^{2}   +  \frac{ 1 }{4 } | F_{\hat{\phi}\hat{\th}}|^{2}   ] . r^{2}   d\sigma^{2} dr^{*} ( t = t_{0} ) \label{defmodifiedenergy} \\
\end{eqnarray*}

\end{proposition}
It will be important in the following that the non explicitly
decaying term on the RHS is a local energy rather than a global energy
as given by a direct application of the Morawetz estimate. To obtain
this we construct a solution of the Yang-Mills equation with compactly
supported data on $t=t_i$ and which coincides with our solution in the
region we are interested in. More precisely let $\hat{F}$ be the curvature associated to gauge transformations applied to the following potential $\hat{A}$,
\bea
\hat{A} &=& \hat{W}(t, r) \tau_{1} d\th  + \hat{W}(t, r) \sin (\th) \tau_{2} d\phi + \cos (\th) \tau_{3} d\phi \label{A-hat}
\eea
where $\hat{W}$ is defined as the solution to the following Cauchy problem
\beaa
 \pa_{t}^2 \hat{W}  - \pa_{r^{*}}^2 \hat{W} -    P  \hat{W}  [ 1 - \hat{W}^{2}  ]  = 0
\eeaa
where
\beaa
\hat{W}  (t=t_{i}, r^{*})=   \hat{\chi} (\frac{2 r^{*}}{t_{i}})    W  (t=t_{i}, r^{*}) 
\eeaa
\beaa
\pa_{t} \hat{W}  (t=t_{i}, r^{*}) = \hat{\chi} (\frac{2 r^{*}}{t_{i}})  \pa_{t} W  (t=t_{i}, r^{*})
\eeaa
and $\hat{\chi}$ is a smooth cut-off function equal to one on $[-1,1]$ and zero outside $[-\frac{3}{2}, \frac{3}{2}]$. 
We have that the expression of $F$ in the Ansatz \eqref{Ansatz}, and the expression of $\hat{F}$ in the Ansatz \eqref{A-hat}, verify the following:
\begin{eqnarray*}
\text{For} \, -\frac{t_{i}}{2},\,  \leq r^{*} \leq \frac{t_{i}}{2}&\quad&
\hat{F}_{\hat{r^{*}}\hat{t} } (t= t_{i}, r^{*}) =
F_{\hat{r^{*}}\hat{t} } (t=t_{i}, r^{*}),\, 
\hat{F}_{\hat{\th}\hat{\phi} } (t= t_{i}, r^{*}) =  F_{\hat{\th}\hat{\phi} } (t=t_{i}, r^{*}),\\
\text{and for}  \; \;  t_{i} \leq t \leq t_{i+1} &\quad& {\der}^{\mu} \hat{F}_{\mu\nu} + [\hat{A}^{\mu}, \hat{F}_{\mu\nu}]= 0. 
\end{eqnarray*}
We note that the Bianchi identities will be satisfied, since in this case, the curvature derives from a potential.

The following lemma controls the energy of $\hat{F}$ in terms of a
local energy of $F$ + decaying terms. It shows in particular that we can
apply the Morawetz estimate to $\hat{F}$ if $t_i$ is sufficiently
large. 
\begin{lemma}
\label{lem4.5}
We have 
\beaa
\notag
E_{\hat{F}}^{(\frac{\pa}{\pa t})} (t=t_i) &=& \int_{r*=-\infty}^{\infty}  \int_{\S^2} ( |  \pa_{t} \hat{W}(t, r) |^2    +   |\pa_{r^*}  \hat{W}(t, r)|^2   +  \frac{N [ \hat{W}^{2} (t, r) -1  ]^2}{2 r^2}  ) dr^* d\si^2 \\
&\les& E_{F}^{(\frac{\pa}{\pa t})}  ( - \frac{3t_{i}}{4} \leq r^* \leq \frac{3t_{i}}{4})  + \frac{ ( E_{F}^{ \# (\frac{\pa}{\pa t} )} ( t= t_{0} )  + E_{F}^{ (\frac{\pa}{\pa t} )} ( t= t_{0} ) +1)^2}{t_i} 
\eeaa
\end{lemma}
\begin{proof}
\beaa
&& |\pa_{r^*}  \hat{W} |^2 (t=t_{i}, r^{*}) \\
&=&| \pa_{r^*} [  \hat{\chi} (\frac{2 r^{*}}{t_{i}})   W  (t=t_{i}, r^{*})  ] |^2\\
&=& |   \frac{2 }{t_{i}}  \hat{\chi}' (\frac{2 r^{*}}{t_{i}}) W   +  \hat{\chi} (\frac{2 r^{*}}{t_{i}})   \pa_{r^*}  W    |^2 \\
&\les&  |\frac{2  }{t_{i}}  \hat{\chi}' (\frac{2 r^{*}}{t_{i}}) W   |^2 +  |  \hat{\chi} (\frac{2 r^{*}}{t_{i}})   \pa_{r^*}  W   |^2 \\
&\les&  |\frac{2 }{t_{i}}   \hat{\chi}' (\frac{2 r^{*}}{t_{i}}) |^2 (W^2  -1) | +  |  \frac{2 }{t_{i}}  \hat{\chi}' (\frac{2 r^{*}}{t_{i}})  |^2 + |  \hat{\chi} (\frac{2 r^{*}}{t_{i}})   \pa_{r^*}  W  |^2 \\
&\les& |\frac{2 }{t_{i}}  \hat{\chi}' (\frac{2 r^{*}}{t_{i}}) |^2 (W^2  -1)^2 +  |\frac{2 }{t_{i}}  \hat{\chi}' (\frac{2 r^{*}}{t_{i}})  |^2 + |  \hat{\chi} (\frac{2 r^{*}}{t_{i}})   \pa_{r^*}  W  |^2 
\eeaa
Hence,
\beaa
\notag
&& \int_{r*=-\infty}^{\infty}  \int_{\S^2}   |\pa_{r^*}  \hat{W}(t, r)|^2 dr^* d\si^2 \les E_{F}^{(\frac{\pa}{\pa t})}  ( - \frac{3t_{i}}{4} \leq r^* \leq \frac{3t_{i}}{4}) \\
 && +  \int_{r*= \frac{-3t_{i}}{4}}^{\frac{-t_{i}}{2}} ( \frac{1}{t_{i}^2} + |\frac{2 }{t_{i}}  \hat{\chi}' (\frac{2 r^{*}}{t_{i}}) |^2 (W^2  -1)^2  )dr^* + \int_{r*=\frac{t_{i}}{2}}^{\frac{3t_{i}}{4}} ( \frac{1}{t_{i}^2} + |\frac{2 }{t_{i}}  \hat{\chi}' (\frac{2 r^{*}}{t_{i}}) |^2 (W^2  -1)^2 ) dr^*   \\
 &\les&E_{F}^{(\frac{\pa}{\pa t})}  ( - \frac{3t_{i}}{4} \leq r^* \leq \frac{3t_{i}}{4}) + \frac{1}{t_{i}} \\
 && + \int_{r*= \frac{-3t_{i}}{4}}^{\frac{-t_{i}}{2}}  |\frac{2 }{t_{i}}  \hat{\chi}' (\frac{2 r^{*}}{t_{i}}) |^2 (W^2  -1)^2 dr^* + \int_{r*=\frac{t_{i}}{2}}^{\frac{3t_{i}}{4}}  |\frac{2 }{t_{i}}  \hat{\chi}' (\frac{2 r^{*}}{t_{i}}) |^2 (W^2  -1)^2 dr^* 
 \eeaa

 We need to control the integrals
 \beaa
\int_{r*= \frac{-3t_{i}}{4}}^{\frac{-t_{i}}{2}}  |\frac{2 }{t_{i}}  \hat{\chi}' (\frac{2 r^{*}}{t_{i}}) |^2 (W^2  -1)^2 dr^* + \int_{r*=\frac{t_{i}}{2}}^{\frac{3t_{i}}{4}}  |\frac{2 }{t_{i}}  \hat{\chi}' (\frac{2 r^{*}}{t_{i}}) |^2 (W^2  -1)^2 dr^*
 \eeaa

We have by a Sobolev inequality, for $r \geq 2m $, $ r \leq r_{1}$,

\begin{eqnarray*}
\Vert (W^2-1)\Vert^2_{L^{\infty}}&\lesssim& \int_{ r^* = -\infty }^{  \infty }
(W^2-1)^2   dr^* +  \int_{ r^* = -\infty }^{  \infty }    (\pa_{r^*}(W^2-1))^2  dr^* \\
&\lesssim& E_{F}^{ \# (\frac{\pa}{\pa t} )} ( t= t_{0} )  +   \int_{ r^* = -\infty }^{  \infty }  
W^2(\pa_{r^*} W)^2  dr^* \\
&\lesssim&  E_{F}^{ \# (\frac{\pa}{\pa t} )} ( t= t_{0} )  +   \int_{ r^* = -\infty }^{  \infty }  (W^2 - 1)(\pa_{r^* } W)^2   dr^*   \\
&\lesssim& (1+\Vert (W^2-1)\Vert_{L^{\infty}})  ( E_{F}^{ \# (\frac{\pa}{\pa t} )} ( t= t_{i} )  + E_{F}^{ (\frac{\pa}{\pa t} )} ( t= t_{0} ) ).
\end{eqnarray*} 
It follows 
\begin{equation*}
\Vert (W^2-1)^2\Vert^2_{L^{\infty}}\lesssim (1+E_{F}^{ \#
  (\frac{\pa}{\pa t} )} ( t= t_{i} )  + E_{F}^{ (\frac{\pa}{\pa t} )}
( t= t_{0} ))^2. 
\end{equation*}

By using the divergence theorem in the region $\{  t_{0} \leq t \leq t_{i} \}$, we obtain

\bea
\notag
&&- I_{F}^{(H)}  (  t_{0} \leq t \leq t_{i} ) ( r \leq r_{1} )    + E^{(H)}_{F} (t_{i})  \\
\notag
&=&  E^{(H)}_{F} (t_{0}) +   I_{F}^{(H)}   (  t_{0} \leq t \leq t_{i} )  ( r \geq r_{1} )   \\ 
\eea
Now we use that $E^{(H)}_{F} (t_{i})$ is equivalent to $E_{F}^{ \#
  (\frac{\pa}{\pa t} )} ( t= t_{i} )$ on $r\le r_1$ and controlled by
$E_{F}^{  (\frac{\pa}{\pa t} )} ( t= t_{i} )$ on $r\ge r_1$.
Due to the positivity of the terms on the left hand side and using the Morawetz estimate \eqref{Morawetz-estimate}, we get 
\bea
\notag
E_{F}^{ \# (\frac{\pa}{\pa t} )} ( t= t_{i} ) \les E_{F}^{ \# (\frac{\pa}{\pa t} )} ( t= t_{0} ) + E_{F}^{  (\frac{\pa}{\pa t} )} ( t= t_{0} )   \label{firstboundenergyfromHforfirstflux}
\eea

On the other hand,
\beaa
0 &\leq& \hat{W}^{2} \leq W^{2} \\
-1 &\leq& \hat{W}^{2} -1 \leq W^{2} -1
\eeaa
Hence,
\beaa
| \hat{W}^{2} -1 |^2 \leq 1+ | W^{2} -1 |^2  \les (1 + E_{F}^{ \# (\frac{\pa}{\pa t} )} ( t= t_{i} )  + E_{F}^{ (\frac{\pa}{\pa t} )} ( t= t_{0} ))^2 
\eeaa
If $\hat{W}^{2} (t, r) -1 \leq 0$, then $(\hat{W}^{2} (t, r) -1 )^2\leq 1$, and
\beaa
\notag
&& \int_{r^*=-\infty}^{\infty}  \int_{\S^2}  \frac{N ( \hat{W}^{2} -1)^{2} (t, r) }{2 r^2}   dr^* d\sigma^2 \\
&\les& \int_{r^*=-\frac{t_{i}}{2}}^{\frac{t_{i}}{2}}  \int_{\S^2}  \frac{N ( W^{2} -1)^{2} }{2 r^2}   dr^* d\sigma^2 + \Vert (\hat{W}^2-1)  \Vert^2_{L^{\infty}} ( \int_{r^*=  \frac{t_{i}}{2} }^{ \infty }  \int_{\S^2}  \frac{N }{2 r^2}   dr^* d\sigma^2 + \int_{r^*= - \infty}^{ \frac{-t_{i}}{2} }  \int_{\S^2}  \frac{N }{2 r^2}   dr^* d\sigma^2 )\\
&\les& E_{F}^{(\frac{\pa}{\pa t})}  ( - \frac{3t_{i}}{4} \leq r^* \leq \frac{3t_{i}}{4}) + \frac{(1+E_{F}^{ \# (\frac{\pa}{\pa t} )} ( t= t_{0} )  + E_{F}^{ (\frac{\pa}{\pa t} )} ( t= t_{0} ) )^2}{ t_{i}}
\eeaa

\end{proof}
{\bf Proof of Proposition
  \ref{estimateforcompactlysupportedintegratedenergyestimate}.}
Recall that 
\begin{eqnarray*}
\lefteqn{\left\vert I^{(H)}_{ F} (  v_{i} \leq v \leq v_{i+1} ) ( w_{i} \leq w \leq
\infty)(r\ge r_1)\right\vert}\\
&=& \left\vert\int \int_{v_i\le v\le v_{i+1}, w_i\le w\le w_{i+1}, r\ge r_1}\int_{\S^{2}} ( [ | F_{\hat{w}\hat{\th}} |^{2} + | F_{\hat{w}\hat{\phi}} |^{2}  ] (    h^{'} - \frac{\mu}{r} h  )   + ( [ | F_{\hat{v}\hat{\th}} |^{2} + | F_{\hat{v}\hat{\phi}} |^{2}  ](  \frac{- h^{'}}{N}  )\right.\\
&+& \left.[  |F_{\hat{v}\hat{w}}|^{2}   +  \frac{1}{4} |
F_{\hat{\phi}\hat{\th}}|^{2}   ]  . \mu [   \frac{-1}{N}  h^{'}  +
\frac{3}{r}   h   ]     ). r^{2} d\sigma^{2} Ndw dv\right\vert\\
&\lesssim& \int_{t_1}^{t_{i+1}}\int_{\R_{r_*}}\int_{S^2}
\chi_{r_1^*\le r_*\le 1.2 r_1^*}(r^*) [   P N^{2} ( |
F_{\hat{w}\hat{\th}} |^{2} +  | F_{\hat{w}\hat{\phi}} |^{2} ) ] d\si^2
dr^{*}dt\\
&+&  \int_{t_1}^{t_{i+1}}\int_{\R_{r_*}}\int_{S^2} \chi_{r_1^*\le
  r_*\le 1.2 r_1^*}(r^*) [\frac{N}{r} (  |F_{\hat{v}\hat{w}}|^{2}    +  | F_{\hat{\phi}\hat{\th}}|^{2}  )   +  P | F_{\hat{v}\hat{\th}} |^{2} +  P | F_{\hat{v}\hat{\phi}} |^{2}]     r^{2}  d\si^2 dr^{*}dt. 
\end{eqnarray*}
By finite speed of propagation this equals for $t_i$ sufficiently large 
\begin{eqnarray}
\label{sthatF}
\lefteqn{\int_{t_1}^{t_{i+1}}\int_{\R_{r_*}}\int_{S^2}
\chi_{r_1^*\le r_*\le 1.2 r_1^*}(r^*) [   P N^{2} ( |
\hat{F}_{\hat{w}\hat{\th}} |^{2} +  | \hat{F}_{\hat{w}\hat{\phi}} |^{2} ) ] d\si^2
dr^{*}dt}\nonumber\\
&+&  \int_{t_1}^{t_{i+1}}\int_{\R_{r_*}}\int_{S^2} \chi_{r_1^*\le
  r_*\le 1.2 r_1^*}(r^*) [\frac{N}{r} (  |\hat{F}_{\hat{v}\hat{w}}|^{2}    +
| \hat{F}_{\hat{\phi}\hat{\th}}|^{2}  )   +  P | \hat{F}_{\hat{v}\hat{\th}} |^{2}
+  P | \hat{F}_{\hat{v}\hat{\phi}} |^{2}]     r^{2}  d\si^2 dr^{*}dt.\nonumber\\ 
\end{eqnarray}

By Lemma \ref{lem4.5} we can apply the Morawetz estimate to
$\hat{F}$. Applying the Morawetz estimate to the term \eqref{sthatF}
and applying again Lemma \ref{lem4.5} gives the Proposition. 
\qed

\subsubsection{Estimate on the bulk and the flux generated from $H$
  near the horizon}
Let us summarize the estimates we have obtained so far :
\begin{proposition}
We have
\bea
\notag
&& - I_{F}^{(H)}  (  v_{i} \leq v \leq v_{i+1} ) ( w_{i} \leq w \leq \infty) ( r \leq r_{1} )   \\
\notag
&& - F_{F}^{(H)} ( v = v_{i+1} ) ( w_{i} \leq w \leq \infty )  - F_{F}^{(H)} ( w = \infty ) ( v_{i} \leq v \leq v_{i+1} )  \\
\notag
&\lesssim& F_{F}^{(\frac{\pa}{\pa t} )} ( w = w_{i} ) ( v_{i} \leq v \leq v_{i+1} )  - F_{F}^{(H)} ( v = v_{i} ) ( w_{i} \leq w \leq \infty ) \\
\notag
&&  +   E^{(\frac{\pa}{\pa t})}_{F}  (  -(0.85)t_{i} \leq r^{*} \leq (0.85)t_{i}  ) (t= t_{i})  + \frac{ (E_{F}^{ \# (\frac{\pa}{\pa t} )} ( t= t_{0} )  + E_{F}^{ (\frac{\pa}{\pa t} )} ( t= t_{0} ) +1)^2}{t_i} 
 \\ \label{estimate3H1}
\eea
\end{proposition}

\begin{proof}\

Let us first recall \eqref{divbase} :

\beaa
&& - I_{F}^{(H)} (  v_{i} \leq v \leq v_{i+1} ) ( w_{i} \leq w \leq \infty) \\
&& - F_{F}^{(H)} ( w = \infty ) ( v_{i} \leq v \leq v_{i+1} )   - F_{F}^{(H)} ( v = v_{i+1} ) ( w_{i} \leq w \leq \infty )  \\
&=& - F_{F}^{(H)} ( w = w_{i} ) ( v_{i} \leq v \leq v_{i+1} ) -   F_{F}^{(H)} ( v = v_{i} ) ( w_{i} \leq w \leq \infty ) 
\eeaa
From \eqref{estimate1H}, we have,
\beaa
- F_{F}^{(H)} ( w = w_{i} ) ( v_{i} \leq v \leq v_{i+1} )  \lesssim  F_{F}^{(\frac{\pa}{\pa t} )} ( w = w_{i} ) ( v_{i} \leq v \leq v_{i+1} )
\eeaa
Thus, we obtain
\bea
\notag
&& - I_{F}^{(H)} (  v_{i} \leq v \leq v_{i+1} ) ( w_{i} \leq w \leq \infty) \\
\notag
&& - F_{F}^{(H)} ( w = \infty ) ( v_{i} \leq v \leq v_{i+1} )   - F_{F}^{(H)} ( v = v_{i+1} ) ( w_{i} \leq w \leq \infty )  \\
&\leq& C F_{F}^{(\frac{\pa}{\pa t} )} ( w = w_{i} ) ( v_{i} \leq v \leq v_{i+1} ) -   F_{F}^{(H)} ( v = v_{i} ) ( w_{i} \leq w \leq \infty )   \label{pseudoestimate3H1}
\eea
(where $C$ is a constant).

From Proposition \ref{estimateforcompactlysupportedintegratedenergyestimate}, we get \eqref{estimate3H1}.

\end{proof}

\subsubsection{Estimate on the energy for observers travelling to the black hole}

\begin{proposition}
\label{prop4.7}
We have
\begin{eqnarray}
\label{estimate4H}
\lefteqn{\inf_{ v_{i} \leq v \leq v_{i+1} }  - F_{F}^{(H)} ( v  ) ( w_{i} \leq w \leq \infty )}\nonumber \\
&\lesssim& \frac{- I_{F}^{(H)} (  v_{i} \leq v \leq v_{i+1} ) ( w_{i} \leq w \leq \infty) ( r \leq r_{1})  }{  (v_{i+1} - v_{i} ) } + \sup_{ v_{i} \leq v \leq v_{i+1} } F_{F}^{(\frac{\pa}{\pa t} )} ( v  ) ( w_{i} \leq w \leq \infty ) ( r \geq r_{1} )\nonumber\\
\end{eqnarray}
\end{proposition}
\begin{proof}\

Using \eqref{60}-\eqref{63} in the region $ r \leq r_{1} $, we get for $v \geq v_{i}$,
\beaa
&& - F_{F}^{(H)} ( v   ) ( w_{i} \leq w \leq \infty ) ( r \leq r_{1} )  \\
& \lesssim&   \int_{w = w_{i}, r \leq r_{1} }^{w = \infty } \int_{\S^{2}} ( [ | F_{\hat{w}\hat{\th}} |^{2} +| F_{\hat{w}\hat{\phi}} |^{2}](  [  \frac{\mu}{r} h - h^{'} ] )  +  [ | F_{\hat{v}\hat{\th}} |^{2} +  | F_{\hat{v}\hat{\phi}} |^{2}]( \frac{ h^{'}}{N}  ) \\
\notag
&& + [ |F_{\hat{v}\hat{w}}|^{2}   +  \frac{1}{4} | F_{\hat{\phi}\hat{\th}}|^{2} ]  . \mu [   \frac{1}{N}  h^{'}  - \frac{3}{r}   h   ]     ). r^{2} d\sigma^{2}N dw 
\eeaa

On the other hand, we have,
\beaa
&& F_{F}^{(\frac{\pa}{\pa t})}  ( v = v_{i} ) ( w_{i} \leq w \leq w_{i+1} ) \\
&=& \int_{w = w_{i}}^{w = w_{i+1}} \int_{\S^{2}}   2 [   |F_{\hat{v}\hat{w}}|^{2}   +  \frac{1}{4} | F_{\hat{\phi}\hat{\th}}|^{2}   +  N  | F_{\hat{w}\hat{\th}} |^{2} + N | F_{\hat{w}\hat{\phi}} |^{2}    ] r^{2} N d\sigma^{2} d w \\
\eeaa

Thus, from the boundedness of $h$, $h'$, we have in $ r \geq r_{1} $,

\beaa
 - F_{F}^{(H)} ( v   ) ( w_{i} \leq w \leq \infty ) ( r \geq r_{1} ) & \lesssim & F_{F}^{(\frac{\pa}{\pa t} )} ( v   ) ( w_{i} \leq w \leq \infty ) ( r \geq r_{1} )
\eeaa

Thus,
\beaa
&& - F_{F}^{(H)} ( v   ) ( w_{i} \leq w \leq \infty ) \\
& \lesssim&   \int_{w = w_{i}, r \leq r_{1} }^{w = \infty } \int_{\S^{2}}  ( [| F_{\hat{w}\hat{\th}} |^{2} +| F_{\hat{w}\hat{\phi}} |^{2}](  [  \frac{\mu}{r} h - h^{'} ] )  + ( [   | F_{\hat{v}\hat{\th}} |^{2} + | F_{\hat{v}\hat{\phi}} |^{2} ]( \frac{ h^{'}}{N}  ) \\
\notag
&& + [ |F_{\hat{v}\hat{w}}|^{2}   +  \frac{1}{4} | F_{\hat{\phi}\hat{\th}}|^{2}  ]  . \mu [   \frac{1}{N}  h^{'}  - \frac{3}{r}   h   ]     ). r^{2} d\sigma^{2} N dw  + F_{F}^{(\frac{\pa}{\pa t} )} ( v   ) ( w_{i} \leq w \leq \infty )( r \geq r_{1} ) 
\eeaa

We have,
\beaa
&& (v_{i+1} - v_{i} ) \inf_{ v_{i} \leq v \leq v_{i+1} } - F_{F}^{(H)} ( v   ) ( w_{i} \leq w \leq \infty ) \\
& \leq&  \int_{v = v_{i}}^{v = v_{i+1}} - F_{F}^{(H)} ( v   ) ( w_{i} \leq w \leq \infty ) dv \\
& \lesssim&   \int_{v = v_{i}, r \leq r_{1} }^{v = v_{i+1}} \int_{w = w_{i} }^{w = \infty } \int_{\S^{2}}  ( [| F_{\hat{w}\hat{\th}} |^{2} +| F_{\hat{w}\hat{\phi}} |^{2}](  [  \frac{\mu}{r} h - h^{'} ] ) + ( [  | F_{\hat{v}\hat{\th}} |^{2} + | F_{\hat{v}\hat{\phi}} |^{2} ]( \frac{ h^{'}}{N}  ) \\
\notag
&& + [ |F_{\hat{v}\hat{w}}|^{2}   +  \frac{1}{4} | F_{\hat{\phi}\hat{\th}}|^{2}  ]  . \mu [   \frac{1}{N}  h^{'}  - \frac{3}{r}   h   ]     ). r^{2} d\sigma^{2} N dw dv  + \int_{v = v_{i}}^{v = v_{i+1}} F_{F}^{(\frac{\pa}{\pa t} )} ( v   ) ( w_{i} \leq w \leq \infty ) ( r \geq r_{1} )  dv  \\
&\lesssim & - I_{F}^{(H)} (  v_{i} \leq v \leq v_{i+1} ) ( w_{i} \leq w \leq \infty) ( r \leq r_{1} ) \\
&&  + (v_{i+1} - v_{i} ) \sup_{ v_{i} \leq v \leq v_{i+1} } F_{F}^{(\frac{\pa}{\pa t} )} ( v  ) ( w_{i} \leq w \leq \infty ) ( r \geq r_{1} )
\eeaa

\end{proof}

\subsubsection{Bounding the bulk term generated from $H$ near the horizon}

\begin{proposition}
\label{prop4.8}
We have
\bea
\notag
0 \leq - I_{F}^{(H)}  (  v_{i} \leq v \leq v_{i+1} ) ( w_{i} \leq w \leq \infty)  ( r \leq r_{1} ) &\lesssim&  (E_{F}^{  (\frac{\pa}{\pa t} )}  + E_{F}^{ \# (\frac{\pa}{\pa t} )} ( t= t_{0} )+1)^2  \label{firstestimateuniformboundonI1I2foreachH1andH2} \\
\eea
where,
\begin{eqnarray}
\label{defmodifiedenergy}
\lefteqn{E_{F}^{ \# (\frac{\pa}{\pa t} )} ( t= t_{0} )}\nonumber \\
&=& \int_{r^{*} = - \infty }^{r^{*} = \infty } \int_{\S^{2}}    [ N | F_{\hat{w}\hat{\th}} |^{2} + N | F_{\hat{w}\hat{\phi}} |^{2}  +  | F_{\hat{v}\hat{\th}} |^{2} + | F_{\hat{v}\hat{\phi}} |^{2}   +  |F_{\hat{v}\hat{w}}|^{2}   +  \frac{ 1 }{4 } | F_{\hat{\phi}\hat{\th}}|^{2}   ] . r^{2}   d\sigma^{2} dr^{*} ( t = t_{0} )\nonumber \\
\end{eqnarray}
\end{proposition}
\begin{proof}\
We have 
\beaa
E^{(H)}_{F} (t)  &=& \int_{r^{*}  = -\infty }^{r^{*} = \infty} \int_{\S^{2}}   [ h N(  | F_{\hat{w}\hat{\th}} |^{2} + | F_{\hat{w}\hat{\phi}} |^{2} )  + h (    |F_{\hat{v}\hat{w}}|^{2}   +  \frac{1}{4 } | F_{\hat{\phi}\hat{\th}}|^{2} ) +  hN (  |F_{\hat{v} \hat{w}}|^{2}   +  \frac{1}{4} | F_{\hat{\phi}\hat{\th}}|^{2} ) \\
&&  +  h (  | F_{\hat{v}\hat{\th}} |^{2} +  | F_{\hat{v}\hat{\phi}} |^{2})  ]   r^{2}  d\sigma^{2}  dr^{*}   
\eeaa

By using the divergence theorem in the region $\{   v \leq v_{0}, t_{0} \leq t \leq \infty, r \leq r_{1} \}$, we obtain

\bea
\notag
&& - I_{F}^{(H)}  (   v \leq v_{0} ) ( t_{0} \leq t \leq \infty) ( r \leq r_{1} )   \\
\notag
&&   - F_{F}^{(H)} ( v = v_{0} ) ( w_{0} \leq w \leq \infty )   - F_{F}^{(H)} ( w = \infty ) ( - \infty \leq v \leq v_{0} )  \\
\notag
&=& E^{(H)}_{F} (t_{0})(r\le r_1) 
\eea
Due to the positivity of the terms on the left hand side, we get
\bea
  - F_{F}^{(H)} ( v = v_{0} ) ( w_{0} \leq w \leq \infty ) \les E^{(H)}_{F} (t_{0})  &\les&  E_{F}^{ \# (\frac{\pa}{\pa t} )} ( t= t_{0} )  \label{firstboundenergyfromHforfirstflux}
\eea

From the divergence theorem and the fact that $\frac{\pa}{\pa t} $ is Killing, it is easy to see that by integrating in a suitable region and using the positivity of the energy we get,

\bea
\notag
F_{F}^{(\frac{\pa}{\pa t} )} ( w = w_{i} ) ( v_{i} \leq v \leq v_{i+1} ) &=& F_{F}^{(\frac{\pa}{\pa t} )} ( w = w_{i} ) ( v_{i} \leq v \leq v_{i+1} ) ( r \geq r_{1}) \\
&\lesssim&  E^{(\frac{\pa}{\pa t})}_{F}(t_i)   \label{forrecurrence1below}
\eea

From \eqref{estimate3H1} and \eqref{forrecurrence1below} we get,
\beaa
\notag
&& - I_{F}^{(H)}  (  v_{0} \leq v \leq v_{1} ) ( w_{0} \leq w \leq \infty) ( r \leq r_{1} )  \\
\notag
&&- F_{F}^{(H)} ( v = v_{1} ) ( w_{0} \leq w \leq \infty ) - F_{F}^{(H)}  ( w = \infty ) ( v_{0} \leq v \leq v_{1} ) \\
\notag
&\lesssim& F_{F}^{(\frac{\pa}{\pa t} )} ( w = w_{0} ) ( v_{0} \leq v \leq v_{1} ) -   F_{F}^{(H)} ( v = v_{0} ) ( w_{0} \leq w \leq \infty ) \\
\notag
&&  +   E^{(\frac{\pa}{\pa t})}_{F}  (  -(0.85)t_{0} \leq r^{*} \leq (0.85)t_{0}  ) (t= t_{0})  + \frac{ (1+ E_{F}^{ \# (\frac{\pa}{\pa t} )} ( t= t_{0} )  + E_{F}^{ (\frac{\pa}{\pa t} )} ( t= t_{0} ) )^2}{t_0}  \\ 
&\lesssim&  (1+E^{(\frac{\pa}{\pa t})}_{F}   + E_{F}^{ \# (\frac{\pa}{\pa t} )} ( t= t_{0} ))^2.
\eeaa

On the other hand, for all $ r \leq r_{1}$, we have by \eqref{60}-\eqref{63},
\bea
h^{'}  - \frac{\mu}{r} h  &\leq&  0 \\
\frac{-h'}{N} &\leq&  0 \\
\mu[\frac{-1}{N}  h^{'}  + \frac{3}{r}   h]   &\leq&  0 
\eea
Thus,
\bea
\notag
&& I^{(H)}_{ F} (  v_{i} \leq v \leq v_{i+1} ) ( w_{i} \leq w \leq \infty)  ( r \leq r_{1} ) \\
\notag
&=& \int_{v = v_{i}, r \leq r_{1} }^{v = v_{i+1}} \int_{w = w_{i}}^{w = \infty } \int_{\S^{2}} ( [ | F_{\hat{w}\hat{\th}} |^{2} + | F_{\hat{w}\hat{\phi}} |^{2}  ] (    h^{'} - \frac{\mu}{r} h  )   + ( [ | F_{\hat{v}\hat{\th}} |^{2} + | F_{\hat{v}\hat{\phi}} |^{2}  ](  \frac{- h^{'}}{N}  ) \\
\notag
&& + [  |F_{\hat{v}\hat{w}}|^{2}   +  \frac{1}{4} | F_{\hat{\phi}\hat{\th}}|^{2}   ]  . \mu [   \frac{-1}{N}  h^{'}  + \frac{3}{r}   h   ]     ). r^{2} d\sigma^{2} Ndw dv \\
&\leq & 0
\eea

Hence, by recurrence from inequality \eqref{estimate3H1}, and using \eqref{forrecurrence1below}, we obtain for all integers $i$
\beaa
\notag
&& - I_{F}^{(H)}  (  v_{i} \leq v \leq v_{i+1} ) ( w_{i} \leq w \leq \infty) ( r \leq r_{1} )  \\
\notag
&&- F_{F}^{(H)} ( v = v_{i+1} ) ( w_{i} \leq w \leq \infty ) - F_{F}^{(H)}  ( w = \infty ) ( v_{i} \leq v \leq v_{i+1} ) \\
\notag
&\lesssim& F_{F}^{(\frac{\pa}{\pa t} )} ( w = w_{i} ) ( v_{i} \leq v \leq v_{i+1} ) -   F_{F}^{(H)} ( v = v_{i} ) ( w_{i} \leq w \leq \infty ) \\
\notag
&&  +   E^{(\frac{\pa}{\pa t})}_{F}  (  -(0.85)t_{i} \leq r^{*} \leq (0.85)t_{i}  ) (t= t_{i})    + \sum_{i=0}^{\infty}\frac{ (1+ E_{F}^{ \# (\frac{\pa}{\pa t} )} ( t= t_{0} )  + E_{F}^{ (\frac{\pa}{\pa t} )} ( t= t_{0} ) )^2}{t_i} \\ 
&\lesssim& (1+E^{(\frac{\pa}{\pa t})}_{F}  (t= t_{i})   + E_{F}^{ \# (\frac{\pa}{\pa t} )} ( t= t_{0} ))^2\\ 
\eeaa

Due to sign of $h$, and the definition of $h$, we have that the terms in each of the integrands on the left hand side are positive, hence, we obtain \eqref{firstestimateuniformboundonI1I2foreachH1andH2}.

\end{proof}

\subsubsection{Decay of the flux of $H$}

\begin{proposition}
For all $v$, let $$w_{0}(v) = v - 2 r_{1}^{*} $$

Let
\bea
v_{+} = \max\{1, v \} \label{definitionvplus}
\eea

We have,

\bea
\notag
 - F_{F}^{(H)} ( v  ) ( w_{0}(v) \leq w \leq \infty )  &\lesssim &  \frac{ [ (E_{F}^{  (\frac{\pa}{\pa t} )}  + E_{F}^{ \# (\frac{\pa}{\pa t} )} ( t= t_{0} )+1)^2 + E_{ F }^{(K) } (t_{0}) ] }{  v_{+}^{}  }  \\\label{estimate6Hflux}
\eea
and,
\bea
\notag
 - F_{F}^{(H)} ( w ) (  v-1 \leq \overline{v} \leq v )  &\lesssim &  \frac{ [ (E_{F}^{  (\frac{\pa}{\pa t} )}  + E_{F}^{ \# (\frac{\pa}{\pa t} )} ( t= t_{0} )+1)^2 + E_{ F }^{(K) } (t_{0}) ] }{  v_{+}^{}  }  \\\label{estimate6Hwequalconstantflux}
\eea
\end{proposition}
For the proof we need several lemmas. In this section we
choose $$t_{i} = (1.1)^{i} t_{0},\, t_0>0.$$

\begin{lemma}

We have,
\bea
 \sup_{ v_{i} \leq v \leq v_{i+1} } F_{F}^{(\frac{\pa}{\pa t} )} ( v  ) ( w_{i} \leq w \leq \infty ) ( r \geq r_{1} ) & \lesssim & \frac{E_{F}^{(K)}(t_{i})}{t_{i}^{2} }  \label{lemmaboudning thesupoftheflux}
\eea

\end{lemma}

\begin{proof}

By integrating in a well chosen region and using the divergence theorem we get that,

\bea
\notag
&&  \sup_{ v_{i} \leq v \leq v_{i+1} } F_{F}^{(\frac{\pa}{\pa t} )} ( v  ) ( w_{i} \leq w \leq \infty ) ( r \geq r_{1} ) \\
\notag
& \lesssim & \int_{r^{*}= r_1^*-0.1 t_i  }^{r^{*} = r_1^*+0.1 t_i   } \int_{\S^{2}} (  N | F_{\hat{w}\hat{\th}} |^{2} + N | F_{\hat{w}\hat{\phi}} |^{2}   + \frac{1}{N}  | F_{\hat{v}\hat{\th}} |^{2} +  \frac{1}{N}  | F_{\hat{v}\hat{\phi}} |^{2}   + |F_{\hat{v}\hat{w}}|^{2}   +  \frac{1}{4} | F_{\hat{\phi}\hat{\th}}|^{2}   ). N r^{2}   d\sigma^{2} dr^{*} (t_{i}) \\ \label{boundingthesup}
\eea
Thus using \eqref{EKoverminvsquaredandwsquared} to estimate
\eqref{boundingthesup} gives \eqref{lemmaboudning thesupoftheflux}.

\end{proof}

By Proposition \ref{prop4.7} and Proposition \ref{prop4.8} we obtain :

\bea
\notag
 \inf_{ v_{i} \leq v \leq v_{i+1} } F_{F}^{(H)} ( v  ) ( w_{i} \leq w
 \leq \infty )  &\lesssim & \frac{1}{ (v_{i+1} - v_{i} ) }  (E_{F}^{  (\frac{\pa}{\pa t} )}  + E_{F}^{ \# (\frac{\pa}{\pa t} )} ( t= t_{0} )+1  )^2  + \frac{E_{F}^{(K)}(t_{i})}{t_{i}^{2} } \\ \label{firstdecay} 
\eea

and thus, there exists a $ v_{i}^{\#} \in [ v_{i}, v_{i+1} ] $ where
the above inequality holds.\\

We have,

\beaa
\notag
v_{i+1} - v_{i} &=& 0.1 t_{i}
\eeaa

Let,
\bea
w_{i}^{\#} = v_{i}^{\#}  - 2 r_{1}^{*}
\eea

Note that $w_i^{\#}\ge w_i$. 

Therefore we have using \eqref{firstdecay} and the positivity of
$-F_F^{(H)}(v_i^{\#})(w_i\le w\le w_i^{\#})$

\beaa
- F_{F}^{(H)} ( v_{i}^{\#}  ) ( w_{i}^{\#} \leq w \leq \infty )  &\lesssim & - F_{F}^{(H)} ( v_{i}^{\#}  ) ( w_{i} \leq w \leq \infty )  \\
&\les& \frac{1}{ t_{i} } (E_{F}^{  (\frac{\pa}{\pa t} )}  + E_{F}^{ \# (\frac{\pa}{\pa t} )} ( t= t_{0} )+1)^2    + \frac{E_{F}^{(K)}(t_{i})}{t_{i}^{2} } \\
\eeaa

From \eqref{estimate3H1}, applied in the region $[w_{i}^{\#},
\infty]\times [ v_{i}^{\#}, v_{i+1}] $, we get due to the positivity of $ -I_{F}^{(H)}  (  v_{i}^{\#} \leq v \leq v_{i+1} ) ( w_{i}^{\#} \leq w \leq \infty) ( r \leq r_{1} ) $,  and $- F_{F}^{(H)} ( w = \infty ) ( v_{i}^{\#} \leq v \leq v_{i+1} )$, that,

\beaa
&&  - F_{F}^{(H)} ( v = v_{i+1} ) ( w_{i}^{\#} \leq w \leq \infty )  \\
&\lesssim& F_{F}^{(\frac{\pa}{\pa t} )} ( w = w_{i}^{\#} ) ( v_{i}^{\#} \leq v \leq v_{i+1} ) -   F_{F}^{(H)} ( v = v_{i}^{\#} ) ( w_{i}^{\#} \leq w \leq \infty ) \\
&& +  E^{(\frac{\pa}{\pa t})}_{F}  (  -(0.85)t_{i} \leq r^{*} \leq (0.85)t_{i}  ) (t= t_{i})  + \frac{ (  1+E_{F}^{ \# (\frac{\pa}{\pa t} )} ( t= t_{0} )  + E_{F}^{ (\frac{\pa}{\pa t} )} ( t= t_{0} ) )^2}{t_i} \\
\eeaa

\begin{lemma}
\label{lem4.11}
\beaa
F_{F}^{(\frac{\pa}{\pa t} )} ( w = w_{i}^{\#} ) ( v_{i}^{\#} \leq v \leq v_{i+1} )  &\les& \frac{E_{F}^{(K)}(t_{i})}{t_{i}^{2} } \\
\eeaa

\end{lemma}

\begin{proof}

By applying the divergence theorem in a well chosen region, we get,

\bea
\notag
F_{F}^{(\frac{\pa}{\pa t} )} ( w = w_{i}^{\#} ) ( v_{i}^{\#} \leq v \leq v_{i+1} )  &\les&  E^{(\frac{\pa}{\pa t})}_{F}  (  r_{1}^{*} \leq r^{*} \leq (0.1)t_{i} + r_{1}^{*}   ) (t= t_{i}) 
\eea

By Proposition \ref{prop3.5} we have 

\bea
\notag
&&  E^{(\frac{\pa}{\pa t})}_{F}  (  r_{1}^{*} \leq r^{*} \leq (0.1)t_{i} + r_{1}^{*}   ) (t= t_{i})  \\
\notag
& \lesssim & \frac{E_{F}^{(K)}(t_{i})}{\min_{r^{*} \in \{   r_{1}^{*}   \leq r^{*} \leq (0.1)t_{i} + r_{1}^{*}   \} } |t_{i} - r^{*}|^{2} }  + \frac{E_{F}^{(K)}(t_{i})}{\min_{r^{*} \in \{  r_{1}^{*}   \leq r^{*} \leq (0.1)t_{i} + r_{1}^{*}   \} } |t_{i} + r^{*}|^{2} } \\
\eea

For $r^{*} \in [   r_{1}^{*}, (0.1)t_{i} + r_{1}^{*}   ]$, and $t_{i}$ large enough,
\beaa
\min_{r^{*} \in \{   r_{1}^{*}   \leq r^{*} \leq (0.1)t_{i} + r_{1}^{*}   \} } |t_{i} - r^{*}|^{2} \geq |(0.9)t_{i} - r_{1}^{*} |^{2} 
\eeaa

Therefore,
\beaa
  E^{(\frac{\pa}{\pa t})}_{F}  (  r_{1}^{*} \leq r^{*} \leq (0.1)t_{i} + r_{1}^{*}   ) (t= t_{i}) & \lesssim&  \frac{E_{F}^{(K)}(t_{i})}{t_{i}^{2} } 
\eeaa
\end{proof}

We now apply again the divergence theorem and obtain using Lemma
\ref{lem4.11}, Proposition \ref{prop4.8} and \eqref{firstdecay}

\beaa
  - F_{F}^{(H)} ( v = v_{i+1} ) ( w_{i}^{\#} \leq w \leq \infty ) &\les& \frac{1}{ t_{i} }  (E_{F}^{  (\frac{\pa}{\pa t} )}  + E_{F}^{ \# (\frac{\pa}{\pa t} )} ( t= t_{0} )+1)^2  + \frac{E_{F}^{(K)}(t_{0})}{t_{i}^{2} } 
\eeaa

and thus,

\bea
\notag
  - F_{F}^{(H)} ( v = v_{i+1} ) ( w_{i+1} \leq w \leq \infty ) &\les&   - F_{F}^{(H)} ( v = v_{i+1} ) ( w_{i}^{\#} \leq w \leq \infty ) \\
\notag
&\les& \frac{  (E_{F}^{  (\frac{\pa}{\pa t} )} (t_{0})   + E_{F}^{ \# (\frac{\pa}{\pa t} )} ( t= t_{0} )+1)^2 + E_{F}^{(K)} (t_{0})  }{ t_{i} } \\
\eea

(due to the positivity of $- F_{F}^{(H)} ( v = v_{i+1} ) ( w_{i}^{\#} \leq w \leq \infty ) $).\\

\subsection{Decay for the middle components} \

\begin{proposition}
Let $v_{+}$ be as defined in \eqref{definitionvplus}, we have for all $r$, such that $2m \leq r \leq r_{1}$,

\beaa
  |F_{\hat{\th} \hat{\phi}} (v, w, \om) | =  |\frac{W^2 (v, w) -1}{r^2} |  &\lesssim & \frac{ E_{1}  }{ \sqrt{v_{+} }  }   
 \eeaa

where,

\beaa
E_{1} &= & [    (1+E_{    F }^{(\frac{\pa}{\pa t}) } (t=t_{0})     + E_{ F }^{ \# (\frac{\pa}{\pa t} )} ( t= t_{0} ) )^2+    E_{ F }^{(K) } (t_{0}) ]^{\frac{1}{2}}  
\eeaa
\end{proposition}

\begin{proof}\

We have by a Sobolev inequality, for $r \geq 2m $, $ r \leq r_{1}$,
\begin{eqnarray*}
\Vert (W^2-1)\Vert^2_{L^{\infty}}&\lesssim& \int_{ \overline{v} = v - 1 }^{ \overline{v} = v }  \int_{\S^{2}}
(W^2-1)^2  d\sigma^{2} d\overline{v}  +  \int_{ \overline{v} = v - 1 }^{ \overline{v} = v }  \int_{\S^{2}}  (\pa_{v}(W^2-1))^2  d\sigma^{2} d\overline{v} \\
&\lesssim& - F_{F }^{(H)} ( w  ) ( v-1 \leq \overline{v} \leq v)  +  \int_{ \overline{v} = v - 1 }^{ \overline{v} = v }  \int_{\S^{2}}  
W^2(\pa_{v} W)^2 d\sigma^{2} d\overline{v}\\
&\lesssim& - F_{F }^{(H)} ( w  ) ( v-1 \leq \overline{v} \leq v)  +  \int_{ \overline{v} = v - 1 }^{ \overline{v} = v }  \int_{\S^{2}}    (W^2 - 1)(\pa_{v} W)^2  d\sigma^{2} d\overline{v}  \\
&\lesssim& (1+\Vert (W^2-1)\Vert_{L^{\infty}}) . (- F_{F }^{(H)} ( w  ) ( v-1 \leq \overline{v} \leq v) )
\end{eqnarray*} 
(where we used what we already proved, and the fact that $r$ is bounded in the region $0< 2m \leq r \leq r_{1}$).
Hence,
\beaa
\Vert (W^2-1)\Vert_{L^{\infty}}&\lesssim& \sqrt{|  F_{F }^{(H)} ( w  ) ( v-1 \leq \overline{v} \leq v) ) |} + |  F_{F }^{(H)} ( w  ) ( v-1 \leq \overline{v} \leq v) ) |
\eeaa

Consequently, we have,

\beaa
|F_{\hat{\th} \hat{\phi}} (v, w, \om) |^{2} &\lesssim&   - F_{F }^{(H)} ( w  ) ( v-1 \leq \overline{v} \leq v)  \\
\eeaa

Thus,
\beaa
 |F_{\hat{\th} \hat{\phi}} (v, w_{0}, \om) |  &\lesssim & \frac{ E_{1} }{ (v_{+} )^{\frac{1}{2}}}  
\eeaa

\end{proof}\

\appendix
\section{Proof of Theorem \ref{ThGEYM}}
\label{AppendixB}

In this appendix, we prove Theorem \ref{ThGEYM}. In the following
$\Vert.\Vert$ always stands for $\Vert.\Vert_{L^2}$, and $\int$ for
$\int_{\R}$. The following lemma collects some useful estimates; it proves in particular that the space
  $\cH^1\times L^2$ is exactly the space of finite energy solutions.
\begin{lemma}
\label{lemA.0}
We have 
\begin{eqnarray}
\label{A.0.1}
\Vert W\Vert_{L^4_P}^4&\lesssim& \int \frac{P}{2}(W^2-1)^2 dr^*+1,\\ 
\label{A.0.2}
\int \frac{P}{2}(W^2-1)^2 dr^*&\lesssim& \Vert W\Vert_{L^4_P}^4+1,\\
\label{A.0.3}
\Vert W\Vert_{L^2_P}&\lesssim& \Vert W\Vert_{L^4_P},\\
\label{A.0.4}
\Vert \sqrt{P}W^2\Vert_{L^{\infty}}&\lesssim&\Vert
W'\Vert_{L^2}^2+\Vert W\Vert_{L^4_P}^2. 
\end{eqnarray}
\end{lemma}
\proof 

We first show \eqref{A.0.1}. We estimate 
\begin{eqnarray*}
\Vert W\Vert_{L^4_P}^4&=&\int PW^4=\int P(W^2-1)(W^2+1)+\int P\\
&=&\int P(W^2-1)^2+2\int P(W^2-1)+\int P\\
&\le&\int P(W^2-1)^2+2\left(\int P\right)^{1/2}\left(\int P
  (W^2-1)^2\right)^{1/2}+\int P\\
&\lesssim& \int \frac{P}{2}(W^2-1)^2+1.
\end{eqnarray*}
Here, we have used the Cauchy-Schwarz inequality in the third line. 

\eqref{A.0.2} follows from
\begin{equation*}
\int \frac{P}{2}(W^2-1)^2\lesssim \int P W^4+\int P. 
\end{equation*}
\eqref{A.0.3} follows from the Cauchy-Schwarz inequality
\begin{equation*}
\int PW^2\le\left(\int P\right)^{1/2}\left(\int PW^4\right)^{1/2}.
\end{equation*}
To show \eqref{A.0.4} we use the Sobolev embedding $H^1(\R)\subset
L^{\infty}(\R)$. This gives 
\begin{eqnarray*}
\Vert \sqrt{P}W^2\Vert_{L^{\infty}}^2&\lesssim&\int PW^4+\int
P^{-1}(P')^2W^4+\int PW^2(W')^2\\
&\lesssim& \Vert W\Vert_{L^4_P}^4+\Vert W'\Vert^2\Vert \sqrt{P}W^2\Vert_{L^{\infty}}.
\end{eqnarray*}
Here we have used $P^{-1}P'\lesssim P$. This quadratic inequality implies the result. 
\qed
\\

We now write the Yang-Mills equation as a first order system. We put
\begin{equation*}
\Psi=(W,\frac{1}{i}\partial_tW). 
\end{equation*}
If $W$ is solution of the Yang-Mills equation, then $\Psi=(\psi_1,\psi_2)$  solves
\begin{equation}
\label{FOYM}
\left\{\begin{array}{rcl} \partial_t\Psi&=&A\Psi+F(\Psi),\\
\Psi(0)&=&(W_0,\frac{1}{i}W_1)=:\Psi_0,\end{array}\right.
\end{equation}
Here,
\begin{equation*}
A=i\left(\begin{array}{cc} 0 & \one \\ -\partial_{r^*}^2 & 0
    \end{array}\right),\quad F(\Psi)=\left(\begin{array}{c} 0 \\
      iP\psi_1(\psi_1^2-1) \end{array}\right). 
\end{equation*}
Let $X=\cH^1\times L^2$. Because of
\eqref{A.0.1} and \eqref{A.0.2} $X$ is exactly the space of finite
energy solutions. Note that the natural domain of $A$ is 
\begin{equation*}
D(A)=\cH^2\times H^1=:Z.
\end{equation*}
\begin{remark}
\begin{enumerate}
\item $X$ is defined as a complex Hilbert space. Nevertheless we are looking
for real solutions of \eqref{YMSW} and therefore for solutions of
\eqref{FOYM} with real first component and purely imaginary second
component. This subspace is of course preserved by the evolution and
we can in the following suppose in our estimates that $\psi_1$ and
$i\psi_2$ are real. 
\item 
Note that in this setting the conserved energy writes
\begin{equation*}
\cE(\Psi)=\int
\vert\psi_2\vert^2+\vert\psi_1'\vert^2+\frac{P}{2}(\psi_1^2-1)^2 dr^*.
\end{equation*}
\end{enumerate}
\end{remark}
Theorem \ref{ThGEYM} will follow from
\begin{theorem}
\label{ThA.1}
Let $\Psi_0\in Z$. Then \eqref{FOYM} has a unique strong solution 
\begin{equation*}
\Psi\in C^1([0,\infty);X)\cap C([0,\infty);Z). 
\end{equation*}
\end{theorem}
{\bf Proof of Theorem \ref{ThGEYM} supposing Theorem \ref{ThA.1}.}
The only point that doesn't follow immediately from Theorem \ref{ThA.1}
is 
\begin{equation*}
\sqrt{P}(W^2-1)\in C([0,\infty);H^1).
\end{equation*} 
We compute 
\begin{equation*}
(\sqrt{P}(W^2-1))'=\frac{1}{2}P^{-1/2}P'(W^2-1)+2\sqrt{P}WW'.
\end{equation*}
As $P^{-1/2}P'\lesssim \sqrt{P}$ the first term is continuous by the
property
\begin{equation*}
\sqrt{P}(W^2-1)\in C^1([0,\infty); L^2).
\end{equation*}
For the second term we compute 
\begin{equation*}
\sqrt{P}((WW')(t)-(WW')(t_0))=\sqrt{P}(W(t)-W(t_0))W'(t_0)+\sqrt{P}W(t)(W'(t)-W'(t_0))
\end{equation*}
and thus 
\begin{eqnarray*}
\Vert \sqrt{P}((WW')(t)-(WW')(t_0)\Vert^2&=&\int
P(W(t)-W(t_0))^2W'^2(t_0)+\int PW^2(t)(W'(t)-W'(t_0))^2\\
&\lesssim& \Vert\sqrt{P}(W(t)-W(t_0))^2\Vert_{L^{\infty}}\Vert
 W'(t_0)\Vert^2\\
&+&\Vert\sqrt{P}W^2(t)\Vert_{L^{\infty}}\Vert W'(t)-W'(t_0)\Vert^2\\
&\lesssim& (\Vert
W'(t)-W'(t_0)\Vert^2+\Vert W(t)-W(t_0)\Vert_{L^4_P}^2)\Vert
W'(t_0)\Vert^2\\
&+&(\Vert W'(t)\Vert^2
+\Vert W(t)\Vert^2_{L^4_P})\Vert
W'(t)-W'(t_0)\Vert^2\rightarrow 0,\, t\rightarrow 0. 
\end{eqnarray*}
Here we have also used \eqref{A.0.4}.
\qed

It therefore remains to show Theorem \ref{ThA.1}. We start by studying the linear part:
\begin{lemma}
$A$ is the generator of a $C^0-$ semigroup on $X$.
\end{lemma}
\proof 
First note that $e^{tA}(\psi_1,\psi_2)$ defined as the
  solution at time $t$ for initial data $(\psi_1,\psi_2)$ sends
  $C_0^{\infty}(\R)\times C_0^{\infty}(\R)$ into itself because of
  finite propagation speed. We want to extend this propagator to
  $X$. First note that 
\begin{equation}
\label{A.2.1}
\Vert e^{tA}(\psi_1,\psi_2)\Vert_{\dot{H}^1\times L^2}=\Vert e^{tA}(\psi_1,\psi_2)\Vert_{\dot{H}^1\times L^2}
\end{equation}
because $-iA$ is selfadjoint on $\dot{H}^1\times L^2$. Now, we have for
$(\psi_1,\psi_2)\in C_0^{\infty}(\R)\times C_0^{\infty}(\R)$
\begin{equation}
\label{A.2.2}
(e^{tA}(\psi_1,
      \psi_2))_1=\frac{1}{2}(\psi_1(r^*+t)+\psi_1(r^*-t)+i\int_{r^*-t}^{r^*+t}\psi_2(s) ds).
\end{equation}
We estimate
\begin{eqnarray}
\label{A.2.3}
\Vert \psi_1(r^*+t)+\psi_1(r^*-t)\Vert_{L^4_P}&\lesssim& \Vert
\psi_1\Vert_{L^4_P},\\
\label{A.2.4}
\int P\left\vert\int_{r^*-t}^{r^*+t}\psi_2(s)ds\right\vert^4dr^*&\le&\int
P\left(\int_{r^*-t}^{r^*+t}\vert \psi_2(s)\vert^2ds\right)^24t^2dr^*\nonumber\\
&\le&4t^2\int P dr^*\Vert\psi_2\Vert^4\lesssim t^2\Vert\psi_2\Vert^4.
\end{eqnarray}
Therefore, $e^{tA}$ extends to a semigroup on $X$. It remains to check
that it is strongly continuous. Again, because of the selfadjointness
of $-iA$, we have for $(\psi_1,\psi_2)\in C_0^{\infty}(\R)\times
C_0^{\infty}(\R)$
\begin{equation*}
\Vert(e^{tA}-\one)(\psi_1,\psi_2)\Vert_{\dot{H}^1\times
  L^2}\rightarrow 0,\, t\rightarrow 0.
\end{equation*}
Using \eqref{A.2.2} and the Lebesgue lemma, we see that for $(\psi_1,\psi_2)\in C_0^{\infty}(\R)\times
C_0^{\infty}(\R)$
\begin{equation*}
\Vert((e^{tA}-\one)(\psi_1,\psi_2))_1\Vert_{L^4_P}\rightarrow 0,\, t\rightarrow 0.
\end{equation*}
Strong continuity follows now from a density argument, using
\eqref{A.2.1}, \eqref{A.2.3} and \eqref{A.2.4}. 

\qed

For the nonlinear part, we need:
\begin{lemma}
$F:X\rightarrow X$ is continuously differentiable. 
\end{lemma}
\proof

We first establish that $F$ sends $X$ into $X$. This follows from the
computation
\begin{equation*}
\Vert F(\Psi)\Vert_X\le\Vert
\sqrt{P}\psi_1\Vert(\Vert\sqrt{P}\psi_1^2\Vert_{L^{\infty}}+1)
\le\Vert \psi_1\Vert_{L^4_P}(\Vert \Psi\Vert_X^2+1). 
\end{equation*}
Here, we have used \eqref{A.0.3} and \eqref{A.0.4}. 
Let $h=(h_1,h_2)\in X$. We compute 
\begin{eqnarray*}
\frac{1}{i}(F(\Psi+h)-F(\psi))&=&\left(\begin{array}{cc} 0 & 0 \\ P(3\psi_1^2-1) &
    0\end{array}\right)h +  \left(\begin{array}{c} 0 \\
    3P\psi_1h_1^2+Ph_1^3\end{array}\right)
\end{eqnarray*}
We first have to check that the matrix on the R.H.S. defines a linear
bounded operator on $X$. We estimate 
\begin{eqnarray*}
\Vert P\psi_1^2h_1\Vert&\le&\Vert
\sqrt{P}\psi_1^2\Vert_{L^{\infty}}\Vert\sqrt{P}h_1\Vert\le\Vert
\Psi\Vert_X^2\Vert h_1\Vert_{L^4_P},\\
\Vert Ph_1\Vert&\le&\Vert\sqrt{P}h_1\Vert\le\Vert h_1\Vert_{L^4_P}.
\end{eqnarray*}
Here, we have used \eqref{A.0.3} and \eqref{A.0.4}. It remains to show
that the second term is of order $\cO(\Vert h\Vert_X^2)$. We estimate
using again \eqref{A.0.3} and \eqref{A.0.4}
\begin{eqnarray*}
\Vert P\psi_1h_1^2\Vert&\le&
\Vert\sqrt{P}h_1^2\Vert_{L^{\infty}}\Vert\sqrt{P}\psi_1\Vert\le \Vert
h\Vert_X^2\Vert \psi_1\Vert_{L^4_P},\\
\Vert Ph_1^3\Vert&\le&\Vert \sqrt{P}h_1\Vert\Vert\sqrt{P}
  h_1^2\Vert_{L^{\infty}}\lesssim \Vert h\Vert_X^3.
\end{eqnarray*} 
Let $\cB(X)$ be the space of bounded linear operators on $X$. It remains to
show that 
\begin{equation*}
\Psi\mapsto \tilde{\cL}(\Psi)=\left(\begin{array}{cc} 0 & 0\\
    3P\psi_1^2-1 & 0 \end{array}\right)
\end{equation*}
is continuous as an application from $X$ to $\cB(X)$. This obviously
follows from the continuity of 
\begin{equation*}
\Psi\mapsto \cL(\Psi)=\left(\begin{array}{cc} 0 & 0\\
    P\psi_1^2 & 0 \end{array}\right).
\end{equation*}
We estimate 
\begin{eqnarray*}
\Vert(\cL(\Psi)-\cL(\Phi))h\Vert_X^2&=&\Vert P(\psi_1^2-\phi_1^2)h_1\Vert_{L^2}^2\\
&\le&\Vert\sqrt{P}(\psi_1-\phi_1)\Vert_{L^2}^2\Vert\sqrt{P}(\psi_1+\phi_1)h_1\Vert_{L^{\infty}}^2\\
&\le&\Vert\Psi-\Phi\Vert_X^2\Vert\sqrt{P}(\psi_1+\phi_1)h_1\Vert_{L^{\infty}}^2.
\end{eqnarray*}
Here we have used \eqref{A.0.3}. By the Sobolev
embedding $H^1(\R)\subset L^{\infty}(\R)$ and the fact that
$P^{-1}(P')^2\lesssim P$ we obtain 
\begin{eqnarray*}
\Vert\sqrt{P}(\psi_1+\phi_1)h_1\Vert_{L^{\infty}}^2&\lesssim&\int
P(\psi_1+\phi_1)^2h_1^2+\int
P(\psi_1'+\phi_1')^2h_1^2+\int P(\psi_1+\phi_1)^2(h_1')^2\\
&\lesssim&\left(\int
  P(\psi_1+\phi_1)^4\right)^{1/2}\left(\int Ph_1^4\right)^{1/2}\\
&+&\Vert\sqrt{P}h^2_1\Vert_{L^{\infty}}^2\int
(\psi_1')^2+(\phi_1')^2 dr^*+\Vert\sqrt{P}(\psi_1+\phi_1)^2\Vert_{L^{\infty}}\int
  (h_1')^2 dr^*\\
&\lesssim&(\Vert\Psi\Vert_X^2+\Vert\Phi\Vert_X^2)\Vert h\Vert_X^2.
\end{eqnarray*}
Here we have used the Cauchy-Schwarz inequality and \eqref{A.0.4}. Summarizing we obtain 
\begin{equation*}
\Vert(\cL(\Psi)-\cL(\Phi))h\Vert_X^2\lesssim\Vert \Psi-\Phi\Vert_X^2(\Vert
\Psi\Vert_X^2+\Vert\Phi\Vert_X^2)\Vert h\Vert_X^2
\end{equation*}
and thus 
\begin{equation*}
\Vert \cL(\Psi)-\cL(\Phi)\Vert_{\cB(X)}\lesssim\Vert\Psi-\Phi\Vert_X(\Vert\Psi\Vert_X+\Vert\Phi\Vert_X),
\end{equation*}
which is the required estimate. 
\qed

We will also need the following a priori estimate
\begin{proposition}
\label{propA.2}
There exists $C>0$ such that for all solutions $\Psi\in
C([0,\infty);Z)\cap C^1([0,\infty);X)$ of \eqref{FOYM} we have
uniformly in $T$:
\begin{equation*}
\Vert \Psi(t)\Vert_X\le C(\Vert \Psi_0\Vert_X+1). 
\end{equation*}
\end{proposition}
\proof 

This follows from the conservation of energy. We have 
\begin{equation*}
\Vert \Psi(t)\Vert_X^4\lesssim \cE(\Psi)+1=\cE(\Psi_0)+1\lesssim (\Vert
\Psi_0\Vert_X^4+1). 
\end{equation*}
\qed

{\bf Proof of Theorem \ref{ThA.1}}
By \cite[Theorem 6.1.5]{Pa} \eqref{FOYM} has a unique strong solution
on some interval $[0,T]$. By \cite[Theorem 6.1.4]{Pa} and the a priori
estimate of Proposition \ref{propA.2}, this solution is global. 
\qed

\end{document}